\documentclass[11pt]{article}

\usepackage[T1]{fontenc}
\usepackage[utf8]{inputenc}
\usepackage{authblk}
\usepackage{euscript}
\usepackage{stmaryrd}
\usepackage{amsmath}
\usepackage{amsthm}
\usepackage{epsfig}
\usepackage{amssymb}\usepackage{amscd}
\usepackage[shortlabels]{enumitem}
\usepackage{pifont}
\usepackage{epic}
\usepackage{color}
\usepackage{orcidlink}

  \usepackage{tikz-cd} 

\usepackage{tikz}
\usetikzlibrary{shapes.geometric, arrows, fit, positioning,calc,decorations.pathmorphing}

\tikzstyle{block} = [rectangle, rounded corners, minimum width=2.5cm, minimum height=1cm, text width=3.5cm, text centered, draw=black]
\tikzstyle{arrow} = [thick,->,>=stealth]

\begingroup\makeatletter\ifx\SetFigFont\undefined%
\gdef\SetFigFont#1#2#3#4#5{%
  \reset@font\fontsize{#1}{#2pt}%
  \fontfamily{#3}\fontseries{#4}\fontshape{#5}%
  \selectfont}%
\fi\endgroup%

\numberwithin{equation}{section}
\numberwithin{equation}{subsection}

\theoremstyle{plain}

\newtheorem{question}{Question}[section]

\newtheorem{thm}[equation]{Theorem}
\newtheorem{cor}[equation]{Corollary}
\newtheorem{lem}[equation]{Lemma}
\newtheorem{prop}[equation]{Proposition}

\theoremstyle{definition}

\newtheorem{remark}[equation]{Remark}

\newtheorem{nota}[equation]{Notation}
\newtheorem{ex}[equation]{Example}

\newtheorem{rem}[equation]{Remark}

\numberwithin{equation}{section}
\numberwithin{equation}{subsection}

\newcommand{ \lk }{ \mbox{lk} }
\newcommand{ \rk }{ \mbox{rk} }
\newcommand{ \imm }{ \mbox{Imm} }
\newcommand{ \emb }{ \mbox{Emb} }

\def\C{\mathbb C}

\def\R{\mathbb R}
\def\Z{\mathbb Z}

\newcommand{\caln}{\mathcal{N}}

\definecolor{darkgreen}{rgb}{0,0.4,0}
\newcommand{\csere}[1]{{#1}}

\title{The boundary of the Milnor fibre and a linking invariant of finitely determined germs}

\author[1]{Gerg\H{o} Pint\'er}
\affil[1]{Department of Theoretical Physics, Institute of Physics,
Budapest University of Technology and Economics\\
M\H{u}egyetem rkp. 3., H-1111 Budapest, Hungary\\
{email: \texttt{pinter.gergo@ttk.bme.hu}}}

\author[2]{Tam\'{a}s Terpai \orcidlink{0000-0002-5707-2668}}
\affil[2]{ELTE E\"otv\"os Lor\'and University, Budapest, Hungary, Institute of Mathematics\\
P\'azm\'any P. s\'et\'any 1/C, H-1117 Budapest, Hungary\\
{email: \texttt{terpai@math.elte.hu}}}

\begin{document}

\maketitle

\begin{abstract}
The image of a finitely determined holomorphic germ $\Phi$ from $\C^2$ to $\C^3$  defines a hypersurface singularity $(X,0)$, which is in general non-isolated. We show that the diffeomorphism type of the boundary of the Milnor fibre $\partial F$ of $X$ is a topological invariant of the germ $\Phi$. We establish a correspondence between the gluing coefficients (so-called vertical indices) used in the construction of $\partial F$ and a linking invariant $L$ of the associated sphere immersion introduced by T. Ekholm and A. Sz\H{u}cs. For this we provide a direct proof of the equivalence of the different definitions of $L$. Since $L$ can be expressed in terms of the cross cap number $C(\Phi)$ and the triple point number $T(\Phi)$ of a stable deformation of $\Phi$, we obtain a relation between these invariants and the vertical indices. This is illustrated on several examples.

\end{abstract}


\pagestyle{myheadings} \markboth{{\normalsize
G. Pint\'er, T. Terpai}}{ {\normalsize Boundary of the Milnor fibre and a linking invariant}}

\section{Introduction}

\subsection{The purpose of the paper}

A finitely determined  holomorphic germ  $ \Phi: ( \C^2, 0) \to ( \C^3, 0) $ is a stable immersion off the origin. The associated stable immersion $\Phi|_{\mathfrak{S}}: \mathfrak{S} \looparrowright S^5$ is the restriction of $\Phi$ to a suitably chosen  3-sphere $ \mathfrak{S}$ around the origin in $ \C^2$.  As it is considered as a $\mathcal{C}^{\infty}$ map, it allows more flexible deformations than the
holomorphic germ $\Phi$. The investigation of the associated immersion creates bridges between immersion theory (differential topology) and local complex singularity
theory (`Milnor fibration package' and `analytic stability package'), and it has several mind-expanding consequences on both areas. For the non-isolated hypersurface singularity $(\mbox{im}(\Phi), 0) \subset (\C^3, 0)$, the boundary of the Milnor fibre is constructed in \cite{NP2} starting from the associated immersion. Recently in \cite{dBM}, the analytic invariants $C(\Phi)$ and $T(\Phi)$ -- the cross cap number and the triple point number of a stable deformation of $\Phi$, respectively -- are proved to be topological invariants. The proof builds on the result that
their combination $C(\Phi)-3T(\Phi)$ is equal to the Ekholm--Sz\H{u}cs linking invariant $L(\Phi|_{\mathfrak{S}})$ of the associated immersion \cite{NP, PS}.

In this paper we create further bridges between the results mentioned above. First we introduce an infinite family of generalizations of $L$ for stable immersions $S^3 \looparrowright \R^5$, then we apply them to finitely determined holomorphic germs. 
We provide a purely topological reformulation of the `vertical indices' used in the construction of the boundary of the Milnor fibre. \csere{Based on this reformulation, we establish a relation between the invariants $C(\Phi)$, $T(\Phi)$ and the vertical indices, and we prove that the vertical indices are topological invariants of $\Phi$, thus the boundary of the Milnor fibre is a topological invariant.}

This paper can be considered as a continuation of \cite{PS} in \csere{the} sense that \cite{PS} and this paper both highlight the hidden role of $L$ in \cite{NP, NP2} and in complex singularity theory. The investigation of finitely determined holomorphic germs from $\C^2$ to $\C^3$ via their associated immersions was started by Andr\'{a}s N\'{e}methi and the first author. In \cite{NP} they established relations between immersion invariants and $C$ and $T$, while in \cite{NP2} they constructed the boundary of the Milnor fibre using the associated immersion. However the appearance of the Ekholm--Sz\H{u}cs invariant $L$ was incidental, since \cite{NP} focuses on the so-called Smale invariant \cite{smale} of the associated immersion, and $L$ appears as an ingredient of the Smale invariant formula of Tobias Ekholm and Andr\'{a}s Sz\H{u}cs.  The proof of the topological \csere{invariance} of $C$ and $T$ in \cite{dBM} motivated Andr\'{a}s S\'{a}ndor and the first author to reprove the correspondence $L(\Phi|_{\mathfrak{S}})=C(\Phi)-3T(\Phi)$ independently from the Smale invariant formulas. Although in \cite{NP2} the invariant $L$ is not mentioned, it turned out that $L$ is closely related to the boundary of the Milnor fibre. This observation leads to a topological reinterpretation of the construction described in \cite{NP2}, as it is presented in this paper.

\subsection{Summary of the results for \texorpdfstring{$L$}{L}} The Ekholm--Sz\H ucs invariant $L(f)$ of a stable immersion $ f: S^3 \looparrowright \R^5$ measures the linking of the image with a copy of the double values, shifted slightly along a suitably chosen normal vector field. 
 In the literature different versions of the definition of $L$ can be found (see \cite{ekholm3, ekholm4, ESz, saeki, juhasz}), which differ from each other in the choice of the normal vector fields. Some of the references use the Seifert framing of the double point curve of $f$ in $S^3$ pushed forward by $df$, while other ones use the homotopically unique normal framing of the immersion $f$ (`global normal framing'). In both cases, there are two vectors at a double value, coming from its two preimages,  and the set of double values is shifted along the sum of the two vectors. The equivalence of the different versions of $L$, i.e. $L_1(f)=-L_2(f)$, is proved recently in \cite{PS}, based on their opposite behavior with respect to regular homotopies. 
 
 What is the relation between the two vector fields, the Seifert framing of the double point set and the global normal framing of the immersion, that implies the equality of the invariants with opposite sign? Motivated by this question, we define infinitely many versions of $L$ for each component of the set of  double values by  using arbitrary normal vector fields for the pushing out. \csere{Then,} we  show that by using the Seifert framing on one preimage component, and the global normal framing on the other one, \csere{the resulting} linking number is $0$. Based on this relation between the two framings we provide a new direct proof for the equality $L_1=-L_2$.

 \subsection{Summary of the results for images of finitely determined germs} Let $ g: ( \C^3, 0) \to ( \C, 0) $ be a holomorphic germ, its zero set $ (X, 0) = (g^{-1} (0), 0) \subset ( \C^3, 0) $ is a hypersurface singularity. The Milnor fibre \cite{milnor} of $ (X, 0)  $ is defined as $F= g^{-1} ( \delta ) \cap B^{6}_{ \epsilon }$ (where the radius $\epsilon$ of the Milnor ball is sufficiently small and $ 0 < \delta \ll \epsilon $). It plays a central role in the study of local singularities. $F$ is always an oriented smooth $4$-manifold (in fact, a complex manifold) with boundary. In the case of isolated singularities (i.e. $ (dg)^{-1}(0) \cap X = \{ 0 \} $), the boundary of the Milnor fibre is diffeomorphic to the link $K=X \cap S^5_{ \epsilon}$. 
For non--isolated hypersurface singularities in $(\C^3,0)$ the situation is
more complicated. The link of $(X, 0)$ is not smooth,
hence the boundary of the Milnor fibre cannot be diffeomorphic to it.
In \cite{NSz} a general algorithm is presented, which constructs the boundary of the Milnor fibre for any non-isolated hypersurface singularity $ (X, 0) \subset ( \C^3, 0) $, although it is rather technical. For particular families
of singularities there are more direct descriptions of the plumbing graph
 of $\partial F$ from the peculiar intrinsic geometry of the germ, see e.g. \cite{MP2} and \cite{dBM2}. 
 
 The image of a finitely determined germ $\Phi: ( \C^2, 0) \to ( \C^3, 0) $ can also be defined as the zero set of a certain germ $ g: ( \C^3, 0) \to ( \C, 0) $, and provides a non-isolated hypersurface singularity $ (X, 0):= (g^{-1}(0), 0)= (\mbox{im} (\Phi), 0) \subset ( \C^3, 0)$. In \cite{NP2}, $ \partial F $ is produced  as a surgery of $ \mathfrak{S}\simeq S^3 $ along the double point locus of the associated immersion $ \Phi|_{\mathfrak{S}} $. The gluing is determined by one integer per each component of the double value set, these numbers are the vertical indices $\mathfrak{vi}_j$. 

 We prove that the sum of all vertical indices is
 \begin{equation}\label{eq:osszeg}
\sum_{j} \mathfrak{vi}_j =  -  \sum_{i \neq k} D_i \cdot D_k - L(\Phi|_{\mathfrak{S}})= -  \sum_{i \neq k} D_i \cdot D_k - C(\Phi) + 3 T( \Phi),
 \end{equation} 
 where $D_i \cdot D_k$ are the intersection multiplicities of the double point curve components of $\Phi$. \csere{Note that the new result of this paper is the first equality of \eqref{eq:osszeg}, and the second one is its straightforward consequence by using the equality $L(\Phi|_{\mathfrak{S}})=C(\Phi)-3T(\Phi)$ proved in \cite{NP, PS}. Also notice that every term in the first and in the third expression is effectively computable, while in contrast $L(\Phi|_{\mathfrak{S}})$ cannot be computed easily from the definition.} The correspondence \eqref{eq:osszeg} is \csere{verified} on several examples \csere{in Subsection~\ref{ss:examp}}. An algebraic proof of the equality \eqref{eq:osszeg} for a special type of germs is provided in \cite{NP2}. \csere{In this special case, in particular, $T(\Phi)=0$ holds, and the original proof shows directly the equality of the left and right sides of equation~\eqref{eq:osszeg} without referring to $L(\Phi|_{\mathfrak{S}})$. The original proof in the special case is summarized in Subsection~\ref{ss:cor1T0}.}

 The independent topological proof for the general case is based on our results for $L$.  We define the `nearby embedded $3$-manifolds' associated to a stable immersion $f:S^3 \looparrowright \R^5$. Each of these manifolds is constructed by a surgery of $S^3$ along the double point locus of $f$, and they differ from each other in the surgery coefficients. Furthermore, these manifolds are related to the newly defined versions of $L$. If $f=\Phi|_{\mathfrak{S}} $ is the immersion associated with $ \Phi $, then these manifolds serve as topological candidates for the boundary of the Milnor fibre, and we characterize the one which really is the boundary of the Milnor fibre. In this way we obtain a new topological description for the vertical indices, which implies equation \eqref{eq:osszeg}. 
 
 \csere{Since each term on the right side of equation \eqref{eq:osszeg} is a topological invariant of $\Phi$, the sum of the vertical indices is also topological by equation \eqref{eq:osszeg}. Furthermore, we prove that the collection $\{  \mathfrak{vi}_j\}$ of all vertical indices is also a topological invariant of $\Phi$.}  Although  the vertical indices can be computed based on their original definition \cite{NP2}\csere{ (see Subsection~\ref{ss:boundmilnor})}, it uses an auxiliary germ, and the topological picture is not clear from it. This is replaced by our new description, which implies the topological invariance of the vertical indices. \csere{(Note that equation \eqref{eq:osszeg} is not used for the proof of the topological invariance of the vertical indices. Instead, the proofs of both statements are based on the new topological description of the vertical indices.)}

\csere{As a consequence, we show that} the diffeomorphism type of $\partial F$ is invariant under topological left-right equivalence of germs. In comparison, in the case of isolated singularities the Milnor fibre $F$ can be recovered (up to a product with $\R$) as the total space of the universal cover of $S^5_\epsilon \setminus K$ and hence $F \times \R$ is a topological invariant. It is known from an older argument of L\^e \cite{Le} that the homeomorphism type of $F \times \R$ is also a topological invariant in the case of a hypersurface singularity. The topological invariance of the Milnor fibre $F$ itself in our setting remains an open question.

\begin{figure}[h]\label{fig:mindmap}
\centering
\begin{tikzpicture}[text centered,node distance=4cm,on grid]

\node[red](stableLabel){for stable immersions};

\node[rectangle,transform canvas={rotate around={90:(-0,-3)}},below left=3mm and 3mm of stableLabel](thisPaper){Results of this paper:};

\node[rectangle,rounded corners,draw,below left=1.5cm and -10mm of stableLabel,align=center](infL){Infinitely many generalizations\\of $L$ corresponding to arbitrary\\ normal framings};

\node[rectangle,rounded corners,draw,below=1.5cm of infL](newL){New proof for $L_1=-L_2$};

\node[rectangle,rounded corners,draw,above right=0mm and 7cm of infL,align=center](nearby){Nearby $3$-manifolds\\corresponding to normal framings};

\path [draw,-latex] (infL)--(newL);

\path [draw,decorate,decoration=snake] (infL)--(nearby);

\node[rectangle,rounded corners,draw,align=center,below right=20mm and 6mm of nearby](topVI){Topological reformulation\\of the vertical indices};

\path [draw,-latex] (infL)--(topVI.north west);
\path [draw,-latex] (nearby)--(topVI);

\node[rectangle,rounded corners,draw,align=center,below left=24mm and 9mm of newL](Milnor){The boundary of the\\Milnor fiber is a\\topological invariant};

\node[darkgreen,below left=16mm and -15mm of Milnor](finLabel){for finitely determined holomorphic germs};

\node[rectangle,rounded corners,draw,align=center,right=45mm of Milnor](invarVI){The collection of the\\vertical indices is a\\topological invariant};

\node[rectangle,rounded corners,draw,align=center,right=45mm of invarVI](sumVI){Equation \eqref{eq:osszeg} for\\the sum of the\\vertical indices};

\path [draw,-latex] (topVI)--(sumVI);
\path [draw,-latex] (topVI)--(invarVI.north east);
\path [draw,-latex] (invarVI)--(Milnor);

\path[red,draw,rounded corners] ($(infL.north west)+(-2mm,2mm)$) -- ($(nearby.north east)+(2mm,4.2mm)$) --($(nearby.south east)+(2mm,-2mm)$)--($(nearby.south west)+(0mm,-2mm)$)--($(newL.south east)+(2mm,-2mm)$)--($(newL.south west)+(-7.6mm,-2mm)$) -- cycle;

\path[darkgreen,draw,rounded corners] ($(Milnor.north west)+(-2mm,2mm)$)--($(invarVI.north)+(0mm,2mm)$)--($(topVI.north west)+(-2mm,2mm)$)--($(topVI.north east)+(2mm,2mm)$)--($(sumVI.south east)+(2mm,-2mm)$)--($(Milnor.south west)+(-2mm,-2mm)$)--cycle;

\node[below right=1cm and 15mm of finLabel](otherPaper){Other related results for finitely determined holomorphic germs:};

\node[below left=1cm and 2.8cm of otherPaper,draw,rectangle,rounded corners](Ltop){ $L=C-3T $ and $L$ is topological \cite{NP} \cite{PS}};

\node[right=7.8cm of Ltop,draw,rectangle,rounded corners](CTinvar){$C$ and $T$ are topological invariants \cite{dBM}};

\path [draw,->,decorate,decoration=snake] (Ltop)--(CTinvar);







\end{tikzpicture}
\caption{\csere{The structure of some of the results related to this paper}}
\end{figure}

\subsection{Structure of the article}
In Section~\ref{s:prelim} we introduce the required notions and previous results. We collect the different versions of the Ekholm--Sz\H{u}cs linking invariant $L$ of stable immersions, we introduce the invariants $C$ and $T$ of finitely determined holomorphic germs, and we state the relation between the three invariants for  immersions associated with finitely determined germs. We introduce the vertical indices and we summarize the construction of the boundary of the Milnor fibre for the image of finitely determined germs.

In Section~\ref{s:decL} we define the family of componentwise generalizations of $L$ for stable immersions. As a consequence we provide a new direct proof for the equality of the different versions of $L$ (with a negative sign).

In Section~\ref{s:nearby} we introduce the family of nearby embedded 3-manifolds for a stable immersion. We establish their relation to the generalizations of $L$.

The content of Section~\ref{s:decL} and Section~\ref{s:nearby} is topological. The results of these sections are applied to finitely determined germs in Section~\ref{s:topref}, where a topological reformulation of the vertical indices is given. Based on this we prove their relation to $L$, that is, equation \eqref{eq:osszeg}, and the topological invariance of the vertical indices. The result is illustrated on several examples.

In Section~\ref{s:que} a final remark and some related questions are presented.

\subsection{Funding}

This work was supported by the Ministry of Culture and Innovation and the National Research, Development and Innovation Office within the Quantum Information National Laboratory of Hungary [grant number 2022-2.1.1-NL-2022-00004 to G.P.]; and the National Research, Development and Innovation Office [grant number K 132146 to G.P.].

\subsection{Acknowledgments}

We thank Andr\'as N\'emethi and Andr\'as Sz\H{u}cs for suggesting the topic of this paper and numerous fruitful discussions on it. We are grateful to J\'ozsef Bodn\'ar, L\'aszl\'o Feh\'er, Roberto Gim\'enez Conejero, David Mond and Andr\'as S\'andor for answering our questions and sharing their knowledge on related topics. GP thanks his physicist colleagues Andr\'as P\'alyi, Gy\"orgy Frank, Zolt\'{a}n Guba, D\'aniel Varjas and J\'anos Asb\'oth for the new inspiration to the singularity theory research, and Bence Csajb\'ok for help with numerical calculations in the examples.

\section{Preliminaries}\label{s:prelim}

\csere{
\subsection{Framings of links}\label{ss:framlink}}

\csere{Here we summarize some basic properties of links in $S^3$ and their normal framings. These properties are used in Subsection~\ref{ss:eszinv} and in Section~\ref{s:decL}, where the `linking calculus' is introduced.}

\csere{Let $C \subset S^3$ be an oriented link, that is, a smooth oriented submanifold of dimension 1, or, equivalently, the image of an embedding $\bigsqcup_{i=1}^{l} S^1 \hookrightarrow S^3$ endowed with the induced orientation. Its components are denoted by $C_i \cong S^1 \subset S^3$, $i=1, \dots, l$. By Alexander duality $H_1(S^3 \setminus C, \Z) \cong \Z^l$ holds for the first homology of the link complement. An isomorphism is given by the linking numbers with the components of $C$. That is, for 1-cycles $B $ in $S^3 \setminus C$ (e.g., $B$ is another link with $B \cap C= \emptyset$), the linking numbers $ \lk_{S^3}(B, C_i)$ depend only on the homology class $[B] \in H_1(S^3 \setminus C, \Z)$ of $B$, and the map $[B] \mapsto \lk_{S^3}(B, C_i)$ is a homomorphism, moreover, it is easy to see that this is an isomorphism. (All discussed linking numbers are considered with respect to the natural orientation of the involved curves and submanifolds.)}

\csere{A \emph{normal framing} $v$ of $C$ is a nowhere vanishing normal vector field $v$ of $C\subset S^3$. More precisely, any vector field $v$ along $C$ tangent to $S^3$ and nowhere tangent to $\gamma$ represents a nonzero section of the normal bundle $\nu(C):=TS^3|_{C}/T C$ of $C \subset S^3$. Slightly abusing the notation we do not distinguish the vector field $v$ and the normal framing $[v]$ represented by $v$. We denote by $v_i=v|_{C_i}$ the restrictions to the individual components.}

\csere{For any normal framing $v$ of $C$ we can consider the link $\widetilde{C}^v$, which is the slightly shifted copy of $C$ along $v$, that is, $\widetilde{C}^v=C+\delta v$, $0< \delta \ll 1$. Obviously, $\widetilde{C}^v \cap C=\emptyset$. We denote by $\widetilde{C}^v_i$ the component corresponding to $C_i$. By the above argument, the homology class $[\widetilde{C}^v] \in H_1(S^3 \setminus C, \Z)$ is completely described by the linking numbers $ \lk_{S^3}(\widetilde{C}^v, C_i)=\sum_{j=1}^l \lk_{S^3}(\widetilde{C}^v_j, C_i)$, $i=1, \dots, l$. However, observe that  $\lk_{S^3}(\widetilde{C}^v_j, C_i)=\lk_{S^3}(C_j, C_i)$ holds for $i \neq j $, in particular, these linking numbers do not depend on the framing $v$. Therefore, the homology class of $\widetilde{C}^v_i$ is completely determined by the integers $\lk_{S^3}(\widetilde{C}^v_i, C_i)$, $i=1, \dots, l$, that is, the linking of the components with their own shifted copies.}

\csere{Two normal framings $v$ and $w$ of $C$ are homotopic if and only if their restrictions $v_i$ and $w_i$ are homotopic for every $i=1, \dots, l$. Up to homotopy, $v_i$ and $w_i$ differ by an element of $\pi_1(SO(2)) \cong \Z$. It is convenient to associate the integers $\lk_{S^3}(\widetilde{C}^v_i, C_i)$ and $\lk_{S^3}(\widetilde{C}^w_i, C_i)$ (or equivalently $\lk_{S^3}(\widetilde{C}^v, C_i)$ and $\lk_{S^3}(\widetilde{C}^v, C_i)$) to $v_i$ and $w_i$, respectively. Then, their difference in $\pi_1(SO(2))$ is equal to $\lk_{S^3}(\widetilde{C}^v_i, C_i)-\lk_{S^3}(\widetilde{C}^w_i, C_i)=\lk_{S^3}(\widetilde{C}^v, C_i)-\lk_{S^3}(\widetilde{C}^v, C_i)$. Therefore, we have the following characterization:}

\csere{
\begin{prop}
    For two normal framings $v$ and $w$ of $C$ the following are equivalent:
    \begin{enumerate}
        \item $v$ and $w$ are homotopic,
        \item $[\widetilde{C}^v]=[\widetilde{C}^w]$ in $H_1(S^3 \setminus C, \Z)$,
        \item $\lk_{S^3}(\widetilde{C}^v, C_i)=\lk_{S^3}(\widetilde{C}^w, C_i)$ for all $i=1, \dots, l$,
        \item $\lk_{S^3}(\widetilde{C}^v_i, C_i)=\lk_{S^3}(\widetilde{C}^w_i, C_i)$ for all $i=1, \dots, l$.
    \end{enumerate}
\end{prop}
}

\csere{The \emph{Seifert framing} of $C$ is a normal framing $v$ of $C$ with $[\widetilde{C}^v]=0 \in H_1(S^3 \setminus C, \Z)$. By the above proposition, the Seifert framing is unique up to homotopy, and it is characterized by the equations $\lk_{S^3}(\widetilde{C}^v, C_i)=0$ for all $i=1, \dots, l$.}

\csere{An \emph{oriented Sefiert surface} $T \subset S^3$ of the oriented link $C \subset S^3$ is an oriented submanifold of dimension 2 with boundary $\partial T=C$. By a classical argument, every link admits an oriented Seifert surface. If an oriented Seifert surface $T \subset S^3$ of an oriented link $ C \subset S^3 $ is fixed, then one has two canonical representatives of the Seifert framing as concrete normal fields. The first one is the outward pointing (orthogonal) normal vector field $v$ of $ C= \partial T \subset T$, the other one is the (orthogonal) normal vector field $u$ of $ T \subset S^3 $ determined by the orientation, restricted to $C$. The rotation by $ \pi /2$ in the normal planes of $ C \subset S^3 $ provides a homotopy between the two vector fields $v$ and $w=u|_C$. The slightly shifted copy $\widetilde{T}^u$ of $T$ along $u$ serves as an oriented Seifert surface of the shifted copy $\widetilde{C}^{w}$ of $C$ along $w=u|_C$. Moreover, $\widetilde{T}^u \cap C_i= \emptyset$ holds for all $i=1, \dots, l$. These observations show that the normal vector fields $w=u|_C$ and $v$ really represent the Seifert framing of $C$.}

\subsection{The Ekholm--Sz\texorpdfstring{\H u}{u}cs invariant of stable immersions}\label{ss:eszinv}
The invariant $L(f)$ of a stable immersion $ f: S^3 \looparrowright \R^5 $ measures the linking of a shifted copy of the double values with the whole image of $f$. Different versions of the definition can be found in the literature, for references see below.  In this paragraph, we review these definitions. Their equivalence is proved in \cite{PS} via their behavior along regular homotopies, and we provide a new direct proof for it in Section~\ref{s:decL}. We present the whole argument in the simplest case, for immersions $S^3 \looparrowright \R^5$, although originally they were introduced for different levels of generality (for other manifolds, higher dimensions) in \cite{ekholm3, ekholm4, ESz, saeki}. This subsection is an extended version of the summary in \cite[2.2.2.]{gtezis}, published also in \cite{PS}.


\csere{First we summarize the different versions of $L$ as they appeared in the above references, but adapted to $f: S^3 \looparrowright \R^5$ stable immersions and adding some necessary comments.}

A stable immersion $ f: S^3 \looparrowright \R^5 $ has only single values and double values with transverse intersection of the two branches. Let 
$ \gamma \subset S^3 $ be the double point locus of $f$,
that is
$ \gamma = \{ p \in S^3 \ | \ \exists p' \in S^3: \ p \neq p' \mbox{ and } f(p)=f(p') \} $.
The locus $ \gamma $ is a closed oriented $1$-manifold, i.e. a link in $S^3$ with possibly multiple components.
The map $ f|_{\gamma}: \gamma \to f( \gamma ) $ is a $2$-fold covering.
$ \gamma $ is endowed with an involution $ \iota: \gamma \to \gamma $ such that $\iota(p) \neq p$ and $ f( p )= f ( \iota (p) ) $ \csere{holds} for all $ p \in \gamma $. 

The first definition of $ L(f)$ is from \cite[6.2.]{ekholm3}. \csere{This definition uses the Seifert framing $v$ of the link $\gamma \subset S^3$. Recall that the Seifert framing is the homotopically unique normal framing characterized by the property that the slightly shifted copy $\widetilde{\gamma}^v$ of $\gamma$ along $v$ represents $[\widetilde{\gamma}^v]=0$ in $H_1 ( S^3 \setminus \gamma , \Z) $.}
Let $ q = f(p) = f ( \iota (p)) $ be a double value of $f$. Then $ w(q) = df_p (v(p)) + df_{ \iota(p)} (v (\iota (p)) ) $ defines a vector field $w$ along $ f( \gamma ) $ that is nowhere tangent to the branches of $f$. In this sense $w$ is a normal vector field of $f$ along $f(\gamma)$.
Let $ \widetilde{f(\gamma)} \subset \R^5 $ be the result of pushing $ f( \gamma) $ slightly along $w$, then $ \widetilde{f(\gamma)} $ and $ f( S^3 ) $ are disjoint. The first invariant is the linking number
\begin{equation*}
L_1 (f) := \lk_{ \R^5 } ( \widetilde{f(\gamma)}, f( S^3) )
\end{equation*}
(or equivalently, $L_1 (f)= [ \widetilde{f(\gamma)} ] \in 
H_1(\R^5 \setminus f(S^3), \Z) \cong \Z $).
Note that Ekholm used a slightly different notation: in \cite[2.2., 6.2.]{ekholm3} our $ L_1 (f) $ is denoted by $ \lk (f) $, and $ L(f) $ is defined as $ \lfloor \lk (f) /3 \rfloor $.

The second definition is \cite[Definition 11.]{ESz}, \cite[Definition 2.2.]{saeki}. It works only with further assumptions, see Remark \ref{re:fuas} below. The normal bundle $ \nu (f) $ of $f$ is trivial, since the oriented rank--$2$ vector bundles over $S^3 $ are classified by $ \pi_2 (SO(2)) =0$. Any two trivializations are homotopic, since their difference represents an element in $ \pi_3 (SO(2)) =0$. Let $ (v_1, v_2) $ be the homotopically unique normal framing of $f$ (`global normal framing'), and at a double value $ q = f(p) = f ( \iota (p)) $ define 
$ u(q) = v_1(p) + v_1 (\iota (p))  $. $u$ is a normal vector field along $f( \gamma)$, and let $ \overline{f(\gamma)} \subset \R^5 $ be the result of pushing $ f( \gamma) $ slightly along $u$. \csere{While $ \overline{f(\gamma)} $ and $ f( S^3 ) $ could in general intersect, under the assumptions of Remark \ref{re:fuas} they are disjoint.} The invariant is the linking number (or equivalently, the homology class) 

\begin{equation*}
L_2 (f) := \lk_{ \R^5 } ( \overline{f(\gamma)}, f( S^3) ) = [ \overline{f(\gamma)} ] \in 
H_1(\R^5 \setminus f(S^3), \Z) \cong \Z \mbox{.}
\end{equation*}

Note that the framing $(v_1, v_2)$ can be replaced by an arbitrary nonzero normal vector field $v$ of $f$, since it can be extended to a framing whose first component is $v$.

\begin{rem}\label{re:fuas}
	Without further assumptions it is possible that $u(q)$ is tangent to one of the branches of $f$, hence it can happen that $ \overline{f(\gamma)} \cap f( S^3 ) \neq \emptyset $. To avoid this problem one has to choose a unit normal vector field $v$, or has to assume that the intersection of the branches is orthogonal, which can be reached by a regular homotopy through stable immersions. This is the case with the third version of $L$ below as well. In Section~\ref{s:decL} we present a more coherent discussion of the different versions of $L$, and we do not assume any further condition for the immersion and for the framings. Instead we avoid using directly the normal framings of $f$ for pushing out, see Subsection~\ref{ss:nontrivpush}.
	\end{rem}

The third definition is in \cite[Definition 4.]{ESz}, see also \cite[4.5., 4.6.]{ekholm4}. Let $v$ be a nonzero normal vector field of $f$ along $ \gamma $, that is, a nowhere zero section of $ \nu (f) |_{ \gamma} $. \csere{Denoting by $ E_0( \nu (f)) $ the total space of the bundle of nonzero normal vectors of $f$, the section $v: \gamma \to E_0( \nu (f))$ defines a homological cycle (a `link') in $ E_0( \nu (f)) $. Let $ [v] $ be the homology class represented by $v$ in $ H_1 ( E_0( \nu (f)), \Z ) \cong \Z $. We define a vector field $u_v$ along $ f( \gamma ) $ that depends on $v$ whose value at the point $ q= f(p) = f ( \iota (p) )$ is defined as $u_v (q) = v(p) + v( \iota (p) ) $. Let $ \overline{f(\gamma)}^{(v)} $ be the result of pushing $ f( \gamma) $ slightly along $u_v$, then $ \overline{f(\gamma)}^{(v)} $ and $ f( S^3 ) $ can be assumed to be disjoint keeping in mind Remark~\ref{re:fuas}.} The invariant is
\begin{equation}\label{eq:Lv} 
L_v (f) := \lk_{ \R^5 } ( \overline{f(\gamma)}^{(v)}, f( S^3) ) -[v]= [ \overline{f(\gamma)}^{(v)} ] -[v] \mbox{,}
\end{equation}
where $ [ \overline{f(\gamma)}^{(v)} ] \in 
H_1(\R^5 \setminus f(S^3), \Z) \cong \Z $.

By \cite[Lemma 4.15.]{ekholm4} $ L_v (f) $ is well-defined, that is, $ L_v (f) $ does not depend on the choice of the normal field $v$. Moreover, if $ v $ is the restriction of a (global) normal vector field of $f $ to $ \gamma $, then $[v]=0$. Indeed, the restriction of the normal field of $f$ to a Seifert surface $T \subset S^3$ of $ \gamma $ results a surface $ \overline{T} \subset E_0( \nu (f)) $, 
whose boundary is the image of $ v : \gamma \to E_0( \nu (f))$. Hence $ L_v (f) = L_2 (f) $.

The invariants $L_1 $, $L_2 $ are equal to each other with opposite sign. This follows from the fact that they behave in an inverse way along regular homotopies, i.e. they change with the same number with opposite sign when a stable regular homotopy steps through first order instabilities: immersions with (1) one triple value (`triple point moves') or (2) a self-tangency (`self-tangency moves'). For definitions we refer  to \cite{ekholm3, ekholm4}. The proof  of Proposition~\ref{pr:Leq} is a result of a discussion with Andr\'{a}s Sz\H{u}cs, and it was published in \cite{PS}. Since one of the goals of this paper is a comprehensive analysis of the different versions of $L$, we include it here as well.

\begin{prop}\label{pr:Leq}
	
	(a) $ L_1 (f) $ and $ L_2 (f) $ are invariants of stable immersions. They change by $ \pm 3 $ under triple point moves and do not change under self tangency moves. In other words: if $f$ and $g$ are regular homotopic stable immersions, $h: S^3 \times [0, 1] \to \R^5 $ is a stable regular homotopy between them, then $\pm (L_i (f) -L_i (g))$ is equal to three times the algebraic number of the triple values of the map $ H: S^3 \times [0, 1] \to \R^5 \times [0, 1]$, $H(x, t)=(h(x, t), t)$.
	
	(b) In the above setup $L_1 (f) -L_1 (g)=-(L_2 (f) -L_2 (g))$.
	
	(c) The three definitions are equivalent:
	\[
	L_1 (f) = -L_2 (f) = -L_v (f).
    \]
\end{prop}

\begin{proof}
	Part (a) is proved for $L_1$ in \cite[Lemma 6.2.1.]{ekholm3} and for $L_2=L_v$ in \cite[Theorem 1.]{ekholm4}. 
	
	For part (b), we compare the change of $L_1$ and $L_2$ through a triple point move. In the proof of \cite[Lemma 6.2.1.]{ekholm3} Ekholm defines a local model of the triple point move where $L_1$ increases by $3$. On the other hand, in  the discussion preceding \cite[Definition 6.3]{ekholm4} he provides a convention to measure the change of $L_2$. If we check this convention on the previous local model, we obtain that $L_2$ decreases by $3$ through that triple point move. Hence $L_1$ and $L_2$ changes in opposite ways at each triple point move.
	
	Using part (a) and part (b), we prove part (c) as follows. 
	Since $ L_1 $ and $L_2 $ changes in opposite way along a regular homotopy, $L_1+L_2$ is a regular homotopy invariant. Moreover
	$ L_1$ and $L_2 $ are additive under connected sum, see \cite[Lemma 5.2., Proposition 5.4.]{ekholm4}, \cite[6.5.]{ekholm3}. It follows that $L_1 + L_2$ defines a homomorphism from the group of the regular homotopy classes $ \imm (S^3, \R^5 ) $ to $ \Z$. If $f: S^3 \hookrightarrow \R^5 $ is an embedding, then $ L_1 (f) = L_2 (f) = 0 $, hence $ L_1 + L_2 $ is $0$ on the $24$-index subgroup $ \emb (S^3, \R^5 )$ of $ \imm ( S^3, \R^5 ) \cong \Z$, cf. \cite{hughes-melvin}. It follows that $L_1 + L_2$ is $0$ for every stable immersion, hence $L_1 = - L_2$. 
\end{proof}

Although $L_1=-L_2$ is proved, a natural question arises: what is the relation between the Seifert framing of the double point curve $\gamma$ and the global normal framing of the immersion $f$, which implies the equivalence of the corresponding versions of $L$? Motivated by this question, $L$ will be  studied  thoroughly in Section~\ref{s:decL}. An infinite family of different versions is introduced componentwise, and a new direct proof for $L_2(f)=-L_1(f)$ will be provided.

We fix the following convention.
\begin{nota}
$L(f):=L_1(f)$.
\end{nota}

\subsection{Invariants of the stabilization of finitely determined germs}\label{ss:invfindet} A holomorphic germ $ \Phi: ( \C^2, 0) \to ( \C^3, 0) $  is \emph{finitely $ \mathcal{A}$-determined} (briefly, \emph{finitely determined}), if there is an integer $k$ such that the $k$-th Taylor polynomial of $ \Phi $ determines it up to left-right equivalence, or equivalently, the $ \mathcal{A}$-codimension of $ \Phi$ is finite.
By Mather--Gaffney criterion \cite{Wall, mond-ballesteros}, $ \Phi $ is finitely determined if and only if its restriction $ \Phi|_{ \C^2 \setminus \{ 0 \}} $ is stable. This means that a sufficiently small representative of $\Phi|_{ \C^2 \setminus \{ 0 \}}$ has only (1) regular simple points and (2) double values with transverse intersection of the regular branches.

The only possible multigerms of a  stable deformation of a holomorphic germ $ \Phi: ( \C^2, 0) \to ( \C^3, 0) $ are (1) regular simple points, (2) double values with transverse intersection of the regular branches, (3) triple values with regular intersection of the regular branches and (4) simple \emph{Whitney umbrella (cross cap)} points, cf. \cite{Whitney}. The Whitney umbrellas and the triple values are isolated points, up to analytic $\mathcal{A}$-equivalence they have local normal forms
\begin{equation*} 
    \mbox{Whitney umbrella (cross cap): } \ (s,t) \mapsto (s^2, st, t) 
\end{equation*}
\begin{equation*} 
    \mbox{Triple value:} 
    \left\{ \begin{array}{ccc} 
     (s_1,t_1)    & \mapsto & (0,s_1,t_1)  \\
     (s_2,t_2)    & \mapsto & (t_2,0,s_2)  \\
     (s_3,t_3)    & \mapsto & (s_3,t_3,0)  \\
    \end{array} \right.
\end{equation*}

The numbers $ C( \Phi ) $ of the cross caps and $ T( \Phi ) $ of the triple values are independent of the stabilization, they are analytic invariants of finitely determined germs $ \Phi $.
Both invariants were  introduced by Mond \cite{Mond0, Mond2}, they can be defined in algebraic way as well, without referring to a stabilization, as follows. 

Let $ C_{alg}( \Phi ) $ be the codimension of the \emph{ramification ideal}, which is the ideal in the local ring $ \mathcal{O}_{( \C^2, 0)}$ generated by the determinants of the $2 \times 2$ minors of the Jacobian matrix of $ \Phi : ( \C^2, 0) \to ( \C^3, 0) $. $ T_{alg}( \Phi ) $ is  the codimension of the second Fitting ideal associated with $ \Phi $ in $ \mathcal{O}_{( \C^3, 0)}$ \cite{mondfitting}. If $ \Phi $ is finitely determined, then both $ C_{alg}( \Phi ) $ and $ T_{alg}( \Phi ) $ are finite, and any stabilization of $ \Phi $ has $ C( \Phi )= C_{alg}( \Phi )$ number of cross caps and $ T( \Phi ) =T_{alg}( \Phi )$ number of triple values.   The invariants $ T$ and $ C$ appear in several different contexts, see for example \cite{Mondvan, mararmulti, nunodouble, slicing, mond-ballesteros, gtezis}.

\subsection{The associated immersion}\label{ss:associmm} Let $ \Phi: (\C^2, 0) \to (\C^3,0) $ be  a finitely determined germ. Such a germ, on the level of links of the spacegerms
$(\C^2,0)$ and $(\C^3,0)$, provides
a stable immersion $ S^3 \looparrowright S^5 $ as follows. The preimage $ \mathfrak{S}:= \Phi^{-1} ( S^5_{ \epsilon })$ of the $5$-sphere $ S^5_{ \epsilon } \subset \C^3 $ around the origin, with a sufficiently small radius $ \epsilon $,  is diffeomorphic to $S^3$. The restriction $ \Phi|_{\mathfrak{S}}: \mathfrak{S} \looparrowright S^5_{ \epsilon} $ is the immersion associated with $ \Phi $. The regular homotopy class of $\Phi|_{\mathfrak{S}} $ is independent of all the choices.
The immersions obtained by different choices are regular homotopic to each other through stable immersions. See \cite[2.1.]{NP} or \cite[Subsection 1.1.2.]{gtezis}.

\csere{While the immersions considered in Subsection~\ref{ss:eszinv} are $S^3 \to \R^5$ and not $S^3 \to S^5$ as needed here, neither the immersions nor their homotopies will generically pass through the point $S^5 \setminus \R^5$ ``at infinity''. Hence the equivalence classes of immersions of $S^3$ to $\R^5$ (up to regular homotopy through stable immersions) can be naturally identified with those of immersions of $S^3$ to $S^5$ (see more details in \cite[Example 2.1.14]{gtezis}). Via this identification the Ekholm--Sz\H{u}cs invariant $L(\Phi|_{\mathfrak{S}})$ of the associated invariant $\Phi|_{\mathfrak{S}}$ can be defined, however, its explicit construction causes some technical difficulties, see \cite[Section 3.2.]{PS}.} 
Since the Ekholm--Sz\H{u}cs invariant $L$ does not change along regular homotopies through stable immersions, $L(\Phi|_{\mathfrak{S}})$ is a well defined invariant of $\Phi$. The correspondence
\begin{equation}\label{eq:LCT}
L(\Phi|_{\mathfrak{S}})=C(\Phi)-3T(\Phi)
\end{equation}
was proved first in \cite{NP}, and a new direct proof is published recently in \cite{PS}. \csere{By \cite[Prop. 3.2.2.]{PS}, $L(\Phi|_{\mathfrak{S}})$ is a topological invariant of $\Phi$, i.e. it is invariant under topological left-right equivalence. Based on this fact,  equation \eqref{eq:LCT} was used in \cite{dBM} to prove the topological invariance of $C$ and $T$.}

\subsection{The image and the double points}\label{ss:imdoub} \csere{ Let $ \Phi: (\C^2, 0) \to (\C^3,0) $ be  a finitely determined germ.} Write $(X,0)$ for $({\rm im}(\Phi),0)$ and
let $ g: ( \C^3, 0) \to ( \C, 0) $ be the reduced equation of $(X,0)$.
 Note that $ (X, 0) $ is a {\it non-isolated} hypersurface singularity, except when $ \Phi $ is a regular map (see \cite{NP}).
 We denote by $ (\Sigma,0) = ( \partial_{x_1} g, \partial_{x_2} g, \partial_{x_3} g)^{-1} (0) \subset (\C^3,0) $
 the \textit{reduced} singular locus of  $(X,0)$
  -- that is the closure of the set of double values of $\Phi$.
 Also, we denote by $(D,0)$ the \textit{reduced}
 double point curve $ \Phi^{-1} (\Sigma ) \subset (\C^2,0) $. The reduced equation of $D$ is
$ d: (\C^2, 0) \to (\C, 0) $.
 (In fact, the finite determinacy of the germ $ \Phi$
 is equivalent with the fact that the double point curve $D$ is reduced;  see e.g. \cite{nunodouble}.)
 
 Let $ \Upsilon \subset S^5_\epsilon$ be the link of $\Sigma$.
It is exactly the  set of double values  of $  \Phi |_{ \mathfrak{S}} $.
Let $ \gamma = \Phi^{-1}(\Upsilon) \subset \mathfrak{S}^3 $ denote the set of double points of
$  \Phi |_{ \mathfrak{S}} $, that is,
$ \gamma \subset  \mathfrak {S}  $ is the link of $D$. All link components are considered with  their natural orientations.

\begin{equation*}
     \begin{array}{ccc}
( \C^2, 0) & \to & ( \C^3, 0) \\
\cup &          & \cup \\
(D, 0) &  \to      & ( \Sigma, 0)  \\
 \cup &          & \cup \\
\gamma=D \cap \mathfrak{S}^3 &  \to      & \Upsilon= \Sigma \cap S_\epsilon^5            
\end{array}
\end{equation*}

 Let $ \gamma = \bigsqcup_{i=1}^l \gamma_i $ be the decomposition of $\gamma$ corresponding to the irreducible decomposition $D= \bigcup_{i=1}^l D_i $ of $D$. The involution on $ D$ pairing the double points induces an order 2 permutation $ \sigma $ of $ \{ 1, \dots , l \} $ such that $ \Phi (\gamma_i)=\Phi (\gamma_{ \sigma(i) }) $.
Let $ J $ be the set of unordered pairs   $\{i, \sigma (i)\}$ ($1\leq i \leq l$). It
serves as the index set of the   components of $ \Upsilon $, that is,
 $ \Upsilon= \bigsqcup_{j \in J} \Upsilon_j$.  Each $\gamma_i$ and $\Upsilon_j$ is diffeomorphic with $S^1$.  If $ i = \sigma(i) $ for some $i$ and $j=\{i\}$, then 
 \begin{equation*} \Phi |_{\gamma_i}: \gamma_i \to \Upsilon_j \end{equation*}
 is a nontrivial double covering of circles. We call these components \emph{twisted}. For the other components $j=\{i, \sigma(i)\}$ with $i \neq \sigma(i)$ the covering 
 \begin{equation*} \Phi |_{\gamma_i \sqcup \gamma_{\sigma(i)}}: \gamma_i \sqcup \gamma_{\sigma(i)} \to \Upsilon_j \end{equation*}  
 is  trivial. These components are called \emph{untwisted}.

\subsection{The boundary of the Milnor fibre}\label{ss:boundmilnor}

Recall that the Milnor fibre $F$ of $(X, 0 )=(g^{-1}(0), 0) \subset (\C^3, 0)$ is defined as
$F= g^{-1} ( \delta ) \cap B^{6}_{ \epsilon }$, where the radius $\epsilon$ of the Milnor ball is sufficiently small (cf. Subsection~\ref{ss:associmm}) and $ 0 < \delta \ll \epsilon $. The Milnor fibre $F$ is  a smooth 4-manifold in $B^6_{ \epsilon }$ with boundary $\partial F \subset S^5_{ \epsilon }$. Since $(X, 0)$ is a non-isolated singularity, its link $K=X \cap S^5_{ \epsilon }$ is not smooth, indeed, it is the image of the associated immersion $\Phi|_{\mathfrak{S}}$. Therefore $\partial F$ and $K$ are definitely different objects. Here we summarize the algorithm from \cite{NP2} to construct $\partial F$ as a (generalized) surgery of $\mathfrak{S}$ along $\gamma$. \csere{Another summary of this construction can be found in \cite{GCP}.}

The construction expressed by formula \eqref{eq:rag} is standard, see e.g. \cite[2.3., 3.4.]{NSz}. Briefly speaking, on a local transverse slice at $q \in \Upsilon \subset S^5_{\epsilon}$ diffeomorphic to $\C^2 $, the wedge of  the two coordinate discs $\{xy=0\}$ has to be replaced by the cylinder $\{xy= \delta\}$ to obtain $\partial F$ from $K$, see Figure~\ref{fig:cyl}.
However the computation of the surgery coefficients called \emph{vertical indices} is special for these singularities. They are introduced in \cite[Section 4.]{NP2} via non-trivial constructions and statements, using an `aid germ' $H$ and the Taylor expansions of $H$ and $g$. We present below a sketch of the definition. One of the goals of this paper is to clarify the topological meaning of these numbers in the context of the Ekholm--Sz\H{u}cs invariant $L$ of the associated immersion, see Section~\ref{s:topref}.

\begin{figure}[h]

\centering
\resizebox{8cm}{8cm}{
\begin{picture}(0,0)%
\includegraphics{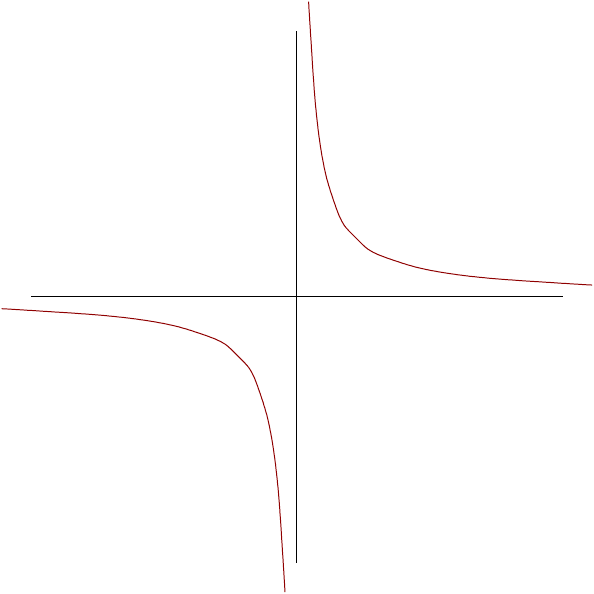}%
\end{picture}%
\setlength{\unitlength}{4144sp}%

\begin{picture}(4524,4524)(1339,-3898)
\put(3916,-376){\makebox(0,0)[lb]{\smash{{\SetFigFont{20}{24.0}{\rmdefault}{\mddefault}{\updefault}{\color[rgb]{.56,0,0}$\partial F$}%
}}}}
\put(4996,-1996){\makebox(0,0)[lb]{\smash{{\SetFigFont{20}{24.0}{\rmdefault}{\mddefault}{\updefault}{\color[rgb]{0,0,0}$K$}%
}}}}
\end{picture}%
}
\caption{The boundary of the Milnor fiber on a transverse slice $\C^2 \subset S^5_{\epsilon}$ at a point $q \in \Upsilon$: the wedge of the two coordinate discs $\{xy=0\}$ is replaced by the cylinder $\{xy= \delta\}$}
\label{fig:cyl}
\end{figure}

A germ $ H: ( \C^3, 0) \to ( \C ,0) $ is called a \emph{transverse section} along $ \Sigma $ if
$ (\Sigma, 0) \subset H^{-1} (0) $ and $H^{-1}(0)$ at any point of $\Sigma \setminus \{ 0 \}$ is smooth and intersects both local
components of $X$ transversally. Transverse sections always exist, we fix one. Clearly $ ( H \circ \Phi)^{-1}(0) \subset ( \C^2, 0) $ decomposes as $ D \cup D_{\sharp} $ with some (not necessarily reduced) curve $ D_{\sharp} $.  Define
\begin{equation*}
\lambda_i = - \sum_{k \not=i } D_k \cdot D_i -D_\sharp \cdot D_i,
\end{equation*}
for $i=1, \dots, l$, where $ C_1 \cdot C_2 $ denotes the intersection multiplicity of the curves
$ (C_1, 0) $ and $ (C_2, 0)  $ at $0\in \C^2$.

Fix
$j=\{i, \sigma(i)\}\in J $, and let  
\begin{equation*} p: (\C, 0) \to (\Sigma_j,0)\subset (\C^3, 0) \mbox{, }
\tau\mapsto p(\tau)  \end{equation*}
be a
parametrization of  $ \Sigma_j $. There exists a splitting $g=g_1 \cdot g_2$ along $\Sigma_j \setminus \{0 \}$ (up to
permutation for twisted components).

\csere{Given a holomorphic function germ $F: (\C^3, P_0) \to (\C, 0)$ defined in a neighborhood of a point $P_0=(x_0, y_0, z_0) \in \C^3$, let $T_k(F)(P_0)$ denote its Taylor polynomial of order $k$ centered at $P_0$ (its variables $x'=x-x_0$, $y'=y-y_0$ and $z'=z-z_0$ are omitted from the notation). With this notation 
\begin{equation}\label{eq:beta}
T_1(H)(p(\tau))=\beta_1(\tau)T_1(g_1)(p(\tau))+\beta_2(\tau)T_1(g_2)(p(\tau)) 
\end{equation}
holds for the first order Taylor polynomial $T_1 (H)$ centered at any point $ 0 \neq p( \tau) \in \Sigma_j $. The product $ \beta_1 \beta_2 $ of the coefficient germs $\beta_1$ and $\beta_2$   is a well-defined meromorphic germ $( \C, 0) \to ( \C, 0)$.  The order (i.e. the smallest negative power) of the Laurent series of $ \beta_1 \beta_2 $ is the \emph{vertical index
	of $ g$ along $ \Sigma_j $ with respect to $ H$}, we denote it by $\mathfrak{v}_j$. Note that both $ \lambda_i $ and $\mathfrak{v}_j$ depend on the choice of $ H$. }

The \emph{vertical index of $ f$ along $ \Sigma_j$} is defined as
\begin{equation*}
\mathfrak{vi}_j := \left\{ \begin{array}{ccc}
\lambda_i  + \lambda_{ \sigma (i)} +  \mathfrak{v}_j  & \mbox{ if } & i \neq \sigma (i) \mbox{,} \\
\lambda_i  + \mathfrak{v}_j  & \mbox{ if } & i = \sigma (i) \mbox{.} \\
\end{array} \right.
\end{equation*}

 \csere{Most of our examples computed in Subsection~\ref{ss:examp} fit into the special case discussed in Subsection~\ref{ss:cor1T0}, which simplifies the computation of the vertical indices and avoids the use of the Taylor series calculations in equation~\eqref{eq:beta}. Even so for the reader's convenience we illustrate this method in Example~\ref{ex:S}. The Taylor series calculations cannot be avoided for more complicated examples like Example~\ref{ex:H}, its full computation can be found in \cite[Subsection 6.7.]{NP2}.}

Let $ N_i $ be a sufficiently small
tubular neighborhood of $ \gamma_i $ in $ \mathfrak{S} $. For each torus $\partial N_i$ we fix two generators in $H_1( \partial N_i, \Z) \cong \Z \oplus \Z $. The oriented \emph{meridian} is a closed curve whose linking number with $\gamma_i $ in $ \mathfrak{S} $ is $1$, and which is the boundary of a disc in $ N_i $. The  \emph{topological longitude} (also called the Seifert framing of $\gamma_i$) is a closed curve that generates $ H_1 (N_i, \Z )$ and whose linking number with $\gamma_i$ is $0$.

\csere{Define the $3$-manifold $ Y$ with torus boundary as the quotient
\begin{equation}\label{eq:3Y}
Y= \frac{[-\pi, \pi] \times S^1 \times [-1,1] }{(-\pi,y,z) \sim (\pi, \bar{y}, -z)},
\end{equation}
where $\bar{y}$ denotes the complex conjugate of $y\in S^1 \subset \C$ (the unit circle). $Y$ is a cylinder bundle over $S^1$  with the coordinate projection $Y \to S^1$ onto the first component $S^1=[-\pi, \pi]/(-\pi \sim \pi)$. The monodromy map $S^1 \times [-1, 1] \to S^1 \times [-1, 1]$ is given by $(y,z) \mapsto (\bar{y},-z)$. For the standard cylinder in $\R^3$ this monodromy map can be interpreted as the rotation by $\pi$ around a symmetry axis of the middle circle, see Figure~\ref{fig:cylmonodr}.}

\csere{With the exception of the two fixed points  $(y,z)=(1,0), (y,z)=(-1,0) \in S^1 \times [-1,1]$ of the monodromy map (on the picture these are the intersection points of the axis and the middle circle), pushing any other point around the base circle $S^1=[-\pi, \pi]/(-\pi \sim \pi)$ returns after two rounds, tracing out the orbit $[-\pi,\pi]\times \{(y,z),(\bar{y},-z)\}$. Define the oriented `topological longitude' $c$ on $ \partial Y $ as the closed orbit of any point.  Define the oriented `meridian' $m$ on $ \partial Y $ as one of the boundary components of an arbitrary cylinder fibre. The different choices for $c$ and for $m$, respectively, are homologous to each other in $\partial Y$.
See \cite[Section 3.]{NP2}, \cite[Section 5.2.]{gtezis}.}

\begin{figure}[h]
\centering
\resizebox{12cm}{12cm}{
\begin{picture}(0,0)%
\includegraphics{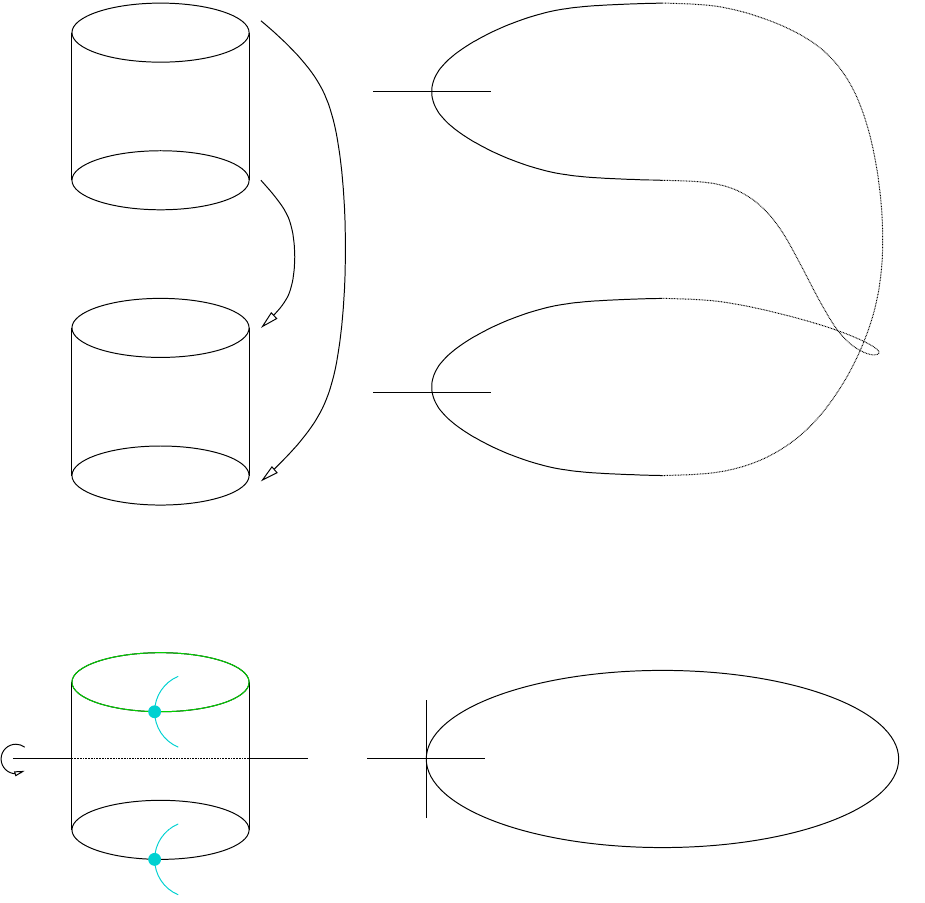}%
\end{picture}%
\setlength{\unitlength}{4144sp}%
\begin{picture}(7044,6862)(802,-6911)
\put(1261,-331){\makebox(0,0)[rb]{\smash{{\SetFigFont{14}{16.8}{\rmdefault}{\mddefault}{\updefault}{\color[rgb]{0,0,0}$|x|\to\infty$}%
}}}}
\put(1261,-1501){\makebox(0,0)[rb]{\smash{{\SetFigFont{14}{16.8}{\rmdefault}{\mddefault}{\updefault}{\color[rgb]{0,0,0}$|y| \to \infty$}%
}}}}
\put(1261,-3706){\makebox(0,0)[rb]{\smash{{\SetFigFont{14}{16.8}{\rmdefault}{\mddefault}{\updefault}{\color[rgb]{0,0,0}$|y| \to \infty$}%
}}}}
\put(1261,-2581){\makebox(0,0)[rb]{\smash{{\SetFigFont{14}{16.8}{\rmdefault}{\mddefault}{\updefault}{\color[rgb]{0,0,0}$|x|\to\infty$}%
}}}}
\put(7606,-601){\makebox(0,0)[lb]{\smash{{\SetFigFont{20}{24.0}{\rmdefault}{\mddefault}{\updefault}{\color[rgb]{0,0,0}$\gamma_i$}%
}}}}
\put(7831,-5461){\makebox(0,0)[lb]{\smash{{\SetFigFont{20}{24.0}{\rmdefault}{\mddefault}{\updefault}{\color[rgb]{0,0,0}$\Upsilon_j$}%
}}}}
\put(4636,-781){\makebox(0,0)[lb]{\smash{{\SetFigFont{14}{16.8}{\rmdefault}{\mddefault}{\updefault}{\color[rgb]{0,0,0}$x=0$}%
}}}}
\put(4636,-3076){\makebox(0,0)[lb]{\smash{{\SetFigFont{14}{16.8}{\rmdefault}{\mddefault}{\updefault}{\color[rgb]{0,0,0}$x=0$}%
}}}}
\put(4591,-5866){\makebox(0,0)[lb]{\smash{{\SetFigFont{14}{16.8}{\rmdefault}{\mddefault}{\updefault}{\color[rgb]{0,0,0}$xy=0$}%
}}}}
\put(991,-5281){\makebox(0,0)[lb]{\smash{{\SetFigFont{20}{24.0}{\rmdefault}{\mddefault}{\updefault}{\color[rgb]{0,.82,0}$m$}%
}}}}
\put(2251,-6856){\makebox(0,0)[lb]{\smash{{\SetFigFont{20}{24.0}{\rmdefault}{\mddefault}{\updefault}{\color[rgb]{0,.82,.82}$c$}%
}}}}
\end{picture}%
} 
\caption{The monodromy of the cylinder bundle $Y$ around a twisted component $\Upsilon_j$: as the point $q \in \Upsilon_j$ goes around, in a transverse slice $\C^2 \subset S^5_{\epsilon}$ the two branches of the associated immersion -- i.e. the coordinates $x$ and $y$ -- interchange, hence the two boundary components of the  cylinder $\{xy=\delta \}$  interchange as well; the meridian $m$ of the torus $\partial Y$ is a boundary component of an arbitrary cylinder fibre, while its longitude $c$ is the orbit of a generic point, rounding two times around $\Upsilon_j$ while $q$ goes around one time}
\label{fig:cylmonodr}
\end{figure}

\begin{thm}[{\cite[Prop. 4.1.2., Theorem 4.5.4., Cor. 4.6.3.]{NP2}, \cite[Prop. 5.3.1., Theorem 5.3.9., Cor. 5.3.11.]{gtezis}}]\label{th:milnper}
	One has an orientation preserving diffeomorphism
	\begin{equation}\label{eq:rag}  \partial F \simeq \left[ \left(\mathfrak{S} \setminus
	\bigcup_{i=1}^l {\rm int}(N_i)  \right)\cup \left(\bigcup_{i = \sigma (i) } Y_i \right) \right] 
	\Bigg / \phi \mbox{,} \end{equation}
	where each $ Y_i $ is diffeomorphic to $Y$, and the gluing $ \phi $ is induced by a
	collection $( \phi_i)_{i=1, \dots, l}$ of diffeomorphisms
	\[ \phi_{i}: \left\{ \begin{array}{ccc} \partial N_i \to -\partial N_{\sigma(i)}
	& \mbox{ if } & i \neq \sigma(i) \\
	\partial N_i \to \partial Y_i \hspace{1cm} & \mbox{ if } & i = \sigma(i). \\
	\end{array} \right.    \]
	Using the (homological)  trivialization of the tori $ \partial N_i $ and $ \partial Y_i $ determined by the pair (meri\-dian, topological longitude), each $ \phi_i $ is given by the matrix                             
	\begin{equation*}
	\left( \begin{array}{cc}
	-1 & \mathfrak{vi}_j \\
	0 & 1 \\
	\end{array} \right)
	\end{equation*}                          
	where $j= \{ i, \sigma (i) \} $.                      
\end{thm}

Using Theorem \ref{th:milnper} it is possible to present $ \partial F$ as a plumbed $3$-manifold, see 
\cite[Theorem 4.8.1.]{NP2}, \cite[Theorem 5.3.18.]{gtezis}.

\begin{cor}[{\cite[Corollary 4.6.4.]{NP2}, \cite[Cor. 5.3.12.]{gtezis}}]
The integer $ \mathfrak{vi}_j $ does not depend on the choice of $H$, thus it is an invariant of $g$ and $ \Sigma_j$.
\end{cor}

In a special case the sum of the vertical indices is determined by the following theorem. 

\begin{thm}[{\cite[Proposition 5.1.1.]{NP2}, \cite[Prop. 5.4.1.]{gtezis}}]
When $ \Phi $ is a corank--$1$ germ with $T( \Phi ) =0$, then 
\begin{equation*}
\sum_{j \in J} \mathfrak{vi}_j = -  \sum_{i \neq k} D_i \cdot D_k - C( \Phi ) .
\end{equation*}
\end{thm}

In fact this formula is a special case of equation~\eqref{eq:osszeg}. A brief summary of its original algebraic proof can be found in Subsection~\ref{ss:examp}. The proof of the general formula \eqref{eq:osszeg} in Section~\ref{s:topref} is based on the topological description of the vertical indices $ \mathfrak{vi}_j $. We also show that the set of vertical indices is a topological invariant of $\Phi$\csere{, and based on this we prove that the diffeomorphism type of $\partial F$ is also a topological invariant of $\Phi$}.

\section{Decomposition and generalization of \texorpdfstring{$L$}{L}}\label{s:decL}

\subsection{Preliminary summary of the section}
In Proposition~\ref{pr:Leq} the different definitions of the Ekholm--Sz\H{u}cs invariant are proved to be equivalent, that is, $L_1(f)=-L_2(f)$ holds for stable immersions $f: S^3 \looparrowright \R^5$. However, the proof is based on the behavior of the two versions of $L$ along regular homotopies, and it does not show any direct relation between them, in particular between the normal vector fields used in their definition. Actually those are different types of framings on the double point curve $\gamma$, since the Seifert framing is a nowhere zero section of the normal bundle $\nu(\gamma)$ of $\gamma \subset S^3$, while the global normal framing of $f$ is a nowhere zero section of $\nu(f)|_{\gamma}$, that is, the normal bundle of $f$ restricted to $\gamma$. \csere{The main goal of this section is to compare the two framings directly by introducing an infinite family of common generalizations. This argument leads to a new proof of $L_1(f)=-L_2(f)$, which also highlights the reason of the negative sign.}

\csere{First we show that for a  pair  of  components $\gamma_i$ and $\gamma_{\sigma(i)}$ of $\gamma$ with the same image $f(\gamma_i)=f(\gamma_{\sigma(i)})=\Upsilon_j$, the sections of the normal bundle $\nu(\gamma_i)$ of $\gamma_i \subset S^3$ can be identified with the sections of the restriction $\nu(f)|_{\gamma_{\sigma(i)}}$ of the normal bundle of $f$ to $\gamma_{\sigma(i)}$ (see Proposition~\ref{pr:isom}). In the rest of Subsection~\ref{ss:dirpr} we introduce integer invariants for these sections and we clarify their relations. To characterize the nowhere zero sections $v$ of $\nu(\gamma_i)$ in the way of Subsection~\ref{ss:framlink},  we introduce the integer $a(v)$ (respectively, $a_i(v)$) that measures its linking with $\gamma$ (respectively, with $\gamma_i$). Alternatively, considering $v$ as a nowhere zero section of $\nu(f)|_{\gamma_{\sigma(i)}}$, we introduce another integer $b(v)$ measuring the twisting of $v$ compared to the global normal framing of $f$ restricted to $\gamma_{\sigma(i)}$.}

\csere{We show that the integer invariants $a$, $a_i$ and $b$ are equivalent characterizations of the sections which differ only in a constant shift by an integer. That is, the differences $a(v)-b(v)$ and $a(v)-a_i(v)$ do not depend on $v$ (see Corollary~\ref{co:diff1} and Proposition~\ref{pr:of}). We show that the difference of $a$ and $a_i$ is an invariant of the link $\gamma$, while the difference of $a$ and $b$ is an invariant of the immersion $f$, associated to the component $\Upsilon_j$. Moreover,  $a(\mathcal{S}_i)=0$  holds for the Seifert framing $\mathcal{S}_i$ and $b(\mathcal{N}_i)=0$ holds for the restriction of the global normal framing $\mathcal{N}_i$, therefore, the difference $a-b$ measures the relation of the two special framings (see Proposition~\ref{pr:of}). Denoting this difference by $\Delta_j(f)$ we show that
\begin{equation*}
\Delta_j(f)=a(v)-b(v)=a(\mathcal{N}_i)=a(\mathcal{N}_{\sigma(i)})=-b(\mathcal{S}_i)=-b(\mathcal{S}_{\sigma(i)}),
\end{equation*}
see Corollary~\ref{co:diff2}.}

\csere{By choosing nowhere zero sections $v$  of $\nu(\gamma_i)$ and $w$ of $\nu(\gamma_{\sigma(i)})$, $\Upsilon_j$ can be shifted slightly in the direction of the vector field $df(v)+df(w)$ to make it disjoint from $f(S^3)$. In Subsection~\ref{ss:nontrivpush} we introduce the integer invariants $c_j(v, w)$ defined as the linking number of the shifted copy of $f(\gamma_i)$ and $f(S^3)$ in $S^5$. In this way we introduce a family of generalizations of $L$ for each component of $f(\gamma)$. (For twisted components, that is, when $i=\sigma(i)$, the construction is a bit more complicated, in this case the `double pushing out' of $\Upsilon_j$ will be defined and will give rise to its linking number $d_j(v, w) $ with the image of $f$ -- see Subsection~\ref{ss:twistedpush}).}

\csere{Then we characterize $c_j(v, w)$ in terms of $a$ and $b$ of $v$, $w$ and $\Delta_j(f)$. The key point is that $c_j(\mathcal{S}_i, \mathcal{N}_i)=0$, see Lemma~\ref{le:kitolt} and Corollary~\ref{co:fels}. We deduce from this observation that
\begin{equation*}
c_j(v, w)=a(v)+a(w)-\Delta_j(f), 
\end{equation*}
see Corollary~\ref{co:ckif}.
We prove a similar result  for the twisted components in Subsection~\ref{ss:twistedpush}.}

\csere{As we will see in Corollary~\ref{co:L12} (see also  Remark~\ref{re:deltas}), the first version $L_1(f)$ of the Ekholm--Sz\H{u}cs invariant $L$ is equal to the sum of the integers 
\begin{equation*}
c_j(\mathcal{S}_{i}, \mathcal{S}_{\sigma(i)})=-\Delta_j(f)
\end{equation*}
for all components, using the modified version for twisted components. Similarly we show that $L_2(f)$ is the sum of the integers
\begin{equation*}
c_j(\mathcal{N}_{i}, \mathcal{N}_{\sigma(i)})=\Delta_j(f),
\end{equation*}
hence we get a direct proof for $L_1=-L_2$ (see Corollary~\ref{co:L1L2equal}).}

In the following we build up the ``linking calculus'' summarized above step by step, finally concluding that $L_1=-L_2$.

\subsection{Normal vector fields along the double point curve} \label{ss:dirpr} 

Recall that the curve of the double points $\gamma \subset S^3$ of a stable immersion $f: S^3 \looparrowright \R^5$   is equipped with the involution $ \iota: \gamma \to \gamma $ such that $f(p)=f(\iota(p))$ holds for all $p \in \gamma$. The following concepts are introduced analogously to the holomorphic case in Subsection~\ref{ss:imdoub}. Let $ \gamma = \bigsqcup_{i \in \{1, \dots , l \}} \gamma_i $ be the irreducible decomposition of $ \gamma $. $f$ induces an order-$2$ permutation $ \sigma: \{ 1, \dots, l \} \to \{ 1, \dots, l \} $ such that $ f ( \gamma_i) = f( \gamma_{ \sigma(i)})=:\Upsilon_{j} $, where $ j= \{ i, \sigma (i) \} $. Let $J$ denote the index set 
$ J= \{ j=\{i, \sigma (i) \} \ | \ i=1 \dots l \} $, then set $ \Upsilon:= f( \gamma) = \bigsqcup_{ j \in J} \Upsilon_j $. If $ i \neq \sigma (i) $, then $ j= \{ i, \sigma (i) \} $ is the index of an \emph{untwisted} component, $ \Upsilon_j $ is trivially covered by $\gamma_i \sqcup \gamma_{\sigma(i)}$. For the \emph{twisted} components $ i= \sigma (i) $, $j=\{ i \} $, and $ f|_{\gamma_i}: \gamma_i \to \Upsilon_j $ is a nontrivial double cover of circles.

Let $\nu(f)$ denote the normal bundle of $f$, and let $\nu(\gamma)$ denote the normal bundle of the inclusion $\gamma \subset S^3$. Both vector bundles are  defined as quotients, namely, $\nu(\gamma)=TS^3|_{\gamma}/T \gamma$ is a bundle over $\gamma$ and $\nu (f) =f^*(T \R^5)/TS^3 $ is a bundle over $S^3$.

\begin{prop}\label{pr:isom}
There is a natural isomorphism of vector bundles 
\begin{equation*}
\Xi:   \iota^* (\nu (f) |_{ \gamma }) \cong   \nu(\gamma) .
\end{equation*}
\end{prop}
 
\begin{proof} At a double value $ q=f(p)=f(\iota(p)) $ the tangent space of $ \R^5 $ decomposes as
\begin{equation}\label{eq:felb} T_q \R^5 \cong df_p(\nu_p ( \gamma )) \oplus df_{ \iota(p)}(\nu_{\iota(p)} ( \gamma)) \oplus T_q \Upsilon .
\end{equation}

According to this, each vector $w \in T_q \R^5 $ can be written in a unique way as 
\begin{equation*} w=w^{ (p)} + w^{(\iota (p))} + w^{(\Upsilon)} .
\end{equation*}
An element of $\nu_{p} (f)  $ is represented by a vector $ v(p) \in T_{q} \R^5$. Define $ \Xi([v(p)])$  as the residue class of
\begin{equation*}
df_{\iota(p)}^{-1}\left(v(p)^{(\iota(p))}\right)  \in T_{\iota(p)} S^3,  
\end{equation*}
in $ \nu_{\iota(p)}(\gamma)$ (see Figure~\ref{fig:map}). It can be shown that $\Xi$ is well defined, the resulting residue class does not depend on the choices of representatives.

\begin{figure}[h]

\centering
\resizebox{15cm}{6cm}{
\begin{picture}(0,0)%
\includegraphics{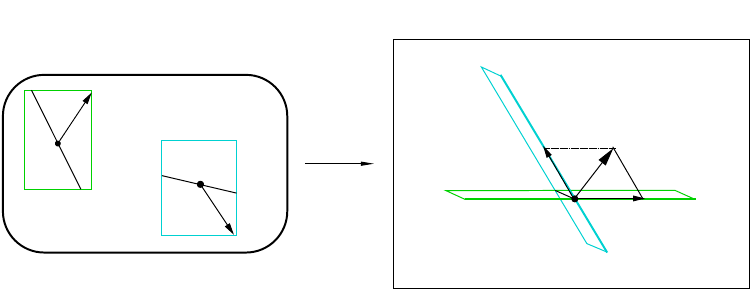}%
\end{picture}%
\setlength{\unitlength}{4144sp}%
\begin{picture}(5722,2192)(-1769,-481)
\put(1505,1097){\makebox(0,0)[lb]{\smash{{\SetFigFont{12}{14.4}{\rmdefault}{\mddefault}{\updefault}{\color[rgb]{1,1,1}$D^4$}%
}}}}
\put(-1211,1295){\makebox(0,0)[lb]{\smash{{\SetFigFont{12}{14.4}{\rmdefault}{\mddefault}{\updefault}{\color[rgb]{0,0,0}$S^3$}%
}}}}
\put(2611, 29){\makebox(0,0)[lb]{\smash{{\SetFigFont{14}{16.8}{\rmdefault}{\mddefault}{\updefault}{\color[rgb]{0,0,0}$q$}%
}}}}
\put(-350,378){\makebox(0,0)[lb]{\smash{{\SetFigFont{12}{14.4}{\rmdefault}{\mddefault}{\updefault}{\color[rgb]{0,0,0}$\iota(p)$}%
}}}}
\put(-1400,464){\makebox(0,0)[lb]{\smash{{\SetFigFont{12}{14.4}{\rmdefault}{\mddefault}{\updefault}{\color[rgb]{0,0,0}$p$}%
}}}}
\put(2324,293){\makebox(0,0)[b]{\smash{{\SetFigFont{12}{14.4}{\rmdefault}{\mddefault}{\updefault}{\color[rgb]{0,0,0}$\Upsilon$}%
}}}}
\put(3017, 17){\makebox(0,0)[lb]{\smash{{\SetFigFont{12}{14.4}{\rmdefault}{\mddefault}{\updefault}{\color[rgb]{0,0,0}$v(p)^{(p)}$}%
}}}}
\put(-1437,915){\makebox(0,0)[lb]{\smash{{\SetFigFont{12}{14.4}{\rmdefault}{\mddefault}{\updefault}{\color[rgb]{0,0,0}$\gamma$}%
}}}}
\put( 57,195){\makebox(0,0)[lb]{\smash{{\SetFigFont{12}{14.4}{\rmdefault}{\mddefault}{\updefault}{\color[rgb]{0,0,0}$\gamma$}%
}}}}
\put(2701,704){\makebox(0,0)[b]{\smash{{\SetFigFont{12}{14.4}{\rmdefault}{\mddefault}{\updefault}{\color[rgb]{0,0,0}$\overset{\Xi}{\mapsfrom}$}%
}}}}
\put(2521,704){\makebox(0,0)[rb]{\smash{{\SetFigFont{12}{14.4}{\rmdefault}{\mddefault}{\updefault}{\color[rgb]{0,0,0}$v(p)^{(\iota(p))}$}%
}}}}
\put(2881,704){\makebox(0,0)[lb]{\smash{{\SetFigFont{12}{14.4}{\rmdefault}{\mddefault}{\updefault}{\color[rgb]{0,0,0}$v(p)$}%
}}}}
\put(1465,1588){\makebox(0,0)[lb]{\smash{{\SetFigFont{14}{16.8}{\rmdefault}{\mddefault}{\updefault}{\color[rgb]{0,0,0}$\mathbb R^5$}%
}}}}
\end{picture}%
} 
\caption{The decomposition of normal vectors of $\Upsilon$ and the map $\Xi$ \csere{drawn in $\R^5$}}
\label{fig:map}
\end{figure}
    
\end{proof}
  
The isomorphism $ \Xi $ induces a bijection between the sections of the bundles $\nu(f)|_{\gamma}$ and $\nu(\gamma)$, that is, between the normal vector fields of $f$ along $ \gamma $ and  the normal vector fields of $ \gamma \subset S^3 $. The corresponding vectors belong to different base points: a normal vector of $f$ at $ p \in \gamma $ determines a normal vector of $ \gamma \subset S^3$ at $ \iota(p) \in \gamma $.

In the following we mostly use nowhere zero sections of the above bundles. For any link $C \subset S^3$, we refer to the nowhere zero sections of its normal bundle $\nu(C)$ as \emph{(normal) framings of $C$}, and to the nowhere zero sections of $\nu(f)|_C$ as \emph{(normal) framings of $f$ along $C$}.

 \begin{rem}\label{re:loc} The framings of $\gamma$ locally look the same in the following sense. A neighborhood $U_q$ of a point $f(p)=f(\iota(p))=q \in \Upsilon_j$ in $\R^5$ can be identified with  $T_q \R^5$, endowed with its decomposition \eqref{eq:felb}. 
Given a framing $v$ of $\gamma_i$ and $w$ of $\gamma_{\sigma(i)}$ and a small positive integer $\delta$, there is a local chart of $(\R^5, q)$, identifying $U_q$ with $ (\C \times \C \times \R, 0)$, such that $ f(S^3)= \{(x,y, \tau) \ | \ xy=0\}$,  $\Upsilon=\{(0,0,\tau)\}$ and
\begin{equation}\label{eq:loktriv}
\begin{aligned}
f(\gamma_i +\delta v)&=\{(1, 0, \tau)\}, \\
f(\gamma_{\sigma(i)}+\delta w)&=\{(0, 1, \tau) \}.
\end{aligned}
\end{equation}
By Proposition~\ref{pr:isom} and the identification $T_q \R^5 \simeq U_q$, the same chart provides local charts of $\nu(f)|_{\gamma_i}$ and  $\nu(f)|_{\gamma_{\sigma(i)}}$. Indeed, $\nu(f)|_{\gamma_i} \cong \nu(\gamma_{\sigma(i)})$ can be identified with the set of elements $(0, y, \tau)$ of $\C \times \C \times \R$, and $\nu(f)|_{\gamma_{\sigma(i)}} \cong \nu(\gamma_i)$ can be identified with $\{(x, 0, \tau)\} \subset \C \times \C \times \R$ 
over $\Upsilon_j \cap U_q$.
\end{rem}

We introduce the following notations for special vector fields. Let $\mathcal{N}$ denote the global normal framing, that is, the homotopically unique normal framing of $f$, which is a nowhere zero section of $\nu(f)$ over $S^3$. Let $\mathcal{N}_i:=\mathcal{N}|_{ \gamma_i } $ denote its restriction. Let $\mathcal{S}$ be the homotopically unique Seifert framing of $ \gamma $, and $ \mathcal{S}_i :=\mathcal{S}|_{ \gamma_i } $. Let $s_i $ be the homotopically unique Seifert framing of $ \gamma_i $. 

Let $v$ be a normal framing of $ \gamma_i \subset S^3 $. We associate two integer invariants to $v$ to measure its global twisting. Define
\csere{\begin{equation}\label{eq:ai}
a_i(v) = \lk_{S^3} (\gamma_i, \gamma_i+ \delta v )
\end{equation}}
with $ 0 < \delta $ sufficiently small, equivalently, $ a_i(v)= [\gamma_i +\delta v] \in H_1 (S^3 \setminus \gamma_i , \Z) \cong \Z  $. The other one is 
\begin{equation*}
a(v) = \lk_{S^3} (\gamma, \gamma_i+ \delta v )
\end{equation*}
 with $ 0< \delta $ sufficiently small. (There is a $ 0 < \delta_0$ such that for $ 0 < \delta < \delta_0 $ the linking numbers do not depend on the choice of $ \delta $.)
 
Let $E(\nu(f))$ be the total space of the bundle $\nu(f)$, and let $E_0(\nu(f))$ denote the complement of the zero section $\underline{0}$. Note that $E(\nu(f))$ is \csere{homotopy} equivalent to $S^3$ and in particular it is  simply connected. Let $w$ be a normal framing of $f$ along $ \gamma_i $. We associate the integer 
\begin{equation*}
 b(w) = \lk_{E(\nu(f))} ( \underline{0}, w),
 \end{equation*}
or, equivalently, $ b(w) = [w] \in H_1 (E_0 ( \nu (f)), \Z ) \cong \Z $  (choosing the isomorphism compatible with the orientations).

Using the isomorphism $ \Xi: \nu(f)|_{ \gamma_{i}} \cong \nu (\gamma_{ \sigma (i)} )  $, we write $ b(v):= b( \Xi^{-1} (v)) $ for a normal framing $v$ of $\gamma_{\sigma(i)}$ as well as $ a_{\sigma(i)} (w) := a_{\sigma(i)} ( \Xi(w)) $ and $ a(w) := a( \Xi (w)) $ for a normal framing $w$ of $f$ along $\gamma_i$.

\begin{prop}\label{pr:of}
\phantom{Here be text, otherwise links break.}
\begin{enumerate}[(a)]
\item $ a_i (s_i) =0 $.
	
\item $ a (\mathcal{S}_i) =0 $.
	
\item $ b (\mathcal{N}_i) =0 $.
	
\item $ a (v) = a_i (v) + \sum_{ k \neq i } \lk_{ S^3} ( \gamma_i, \gamma_k) $, if $v$ is a normal framing of $ \gamma_i \subset S^3 $.
	
\item $ a_i (\mathcal{S}_i) = - \sum_{k \neq i } \lk_{ S^3} ( \gamma_i, \gamma_k) $.
	
\item $ a (s_i) = \sum_{k \neq i } \lk_{ S^3} ( \gamma_i, \gamma_k) $.
 \end{enumerate}
\end{prop}

\begin{proof}
	(a), (b): Slightly pushing out an oriented link $ C \subset S^3$ from its oriented Seifert surface $T \subset S^3$ in the direction of the outward pointing normal vector field of $C= \partial T \subset T$ produces a link $C' \subset S^3$ which is disjoint from $T$. Hence $ \lk_{ S^3} (C, C')=0$.
	
	(c) The image of the restriction of $ \caln $ to a Seifert surface of $ \gamma_i$ is a surface in $ E( \nu(f)) $ which is disjoint from $ \underline{0} $ and whose boundary is $ \caln_i $.
	
	(d) \begin{equation*} a(v)= \lk_{S^3} (\gamma, \gamma_i+ \delta v )= \sum_{k=1}^l \lk_{S^3} (\gamma_k, \gamma_i+ \delta v ) = a_i(v)+\sum_{k \neq i} \lk_{S^3} (\gamma_k, \gamma_i+ \delta v ),\end{equation*} 
  and 
 \begin{equation*} \lk_{S^3} ( \gamma_k, \gamma_i+ \delta v)= \lk_{S^3} ( \gamma_k, \gamma_i) \end{equation*}
 holds for $k \neq i$ with the choice of a sufficiently small $ \delta$.	
 
	(e) and (f) follow from (a), (b) and (d).
\end{proof}

For a normal framing $v$ of  $ \gamma_i \subset S^3 $ let $v^+ $ be the homotopically unique normal framing of $ \gamma_i \subset S^3 $ for which $ a(v^+)=a(v) + 1 $ holds. $v^+$ can be constructed by modifying $v$ over an arbitrarily small part of $\gamma_i$ by adding one total twist to $v$ (see Figure~\ref{fig:vplus}). Since the modification is local on $\gamma_i$, by Remark~\ref{re:loc} it can be considered in the canonical local coordinates that satisfy \eqref{eq:loktriv}. Therefore, interpreting $v$ as a framing of $f$ along $\gamma_{\sigma(i)}$, the total twist increases its linking number with the zero section by 1. Hence we have

\begin{figure}[h]

\centering

\resizebox{15cm}{6cm}{
\begin{picture}(0,0)%
\includegraphics{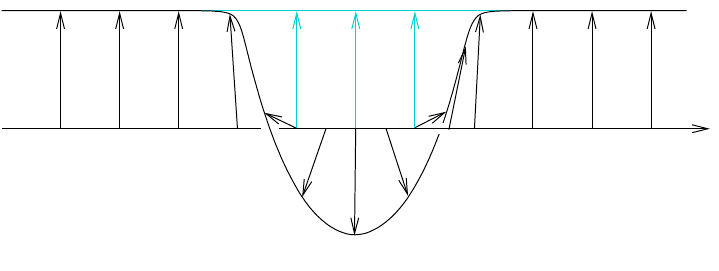}%
\end{picture}%
\setlength{\unitlength}{4144sp}%
\begin{picture}(5472,2061)(-2711,-244)
\put(-107,-189){\makebox(0,0)[lb]{\smash{{\SetFigFont{12}{14.4}{\rmdefault}{\mddefault}{\updefault}{\color[rgb]{0,0,0}$v^+$}%
}}}}
\put(2746,794){\makebox(0,0)[lb]{\smash{{\SetFigFont{12}{14.4}{\rmdefault}{\mddefault}{\updefault}{\color[rgb]{0,0,0}$\gamma_i$}%
}}}}
\put(2746,1694){\makebox(0,0)[lb]{\smash{{\SetFigFont{12}{14.4}{\rmdefault}{\mddefault}{\updefault}{\color[rgb]{0,.82,.82}$v$}%
}}}}
\end{picture}%
}
\caption{Adding an extra twist in $v^+$}
\label{fig:vplus}
\end{figure}

\begin{prop}
$b(v^+)=b(v)+1$.
\end{prop}

\begin{cor}\label{co:diff1}
The difference $a(v)-b(v)$ does not depend on the framing $v$. Moreover,
\begin{equation*}
a(v)-b(v)=a(\mathcal{N}_i)
=
-b(\mathcal{S}_i)
\end{equation*}
holds for any $v$, which can be either a normal framing of $\gamma_i$ or a normal framing of $f$ along $\gamma_{\sigma(i)}$.
\end{cor}

\subsection{Pushing out of untwisted components}\label{ss:nontrivpush}
We define a family of generalizations of $L_1(f)$ and $L_2(f)$ for each component of $ \Upsilon$ and for all pairs of normal vector fields. According to the isomorphism $ \Xi $, the vector fields can be either framings of the components of $ \gamma $ in $ S^3$, or framings of $f$ along the components of $ \gamma $. The pushing out by a pair of vector fields is defined separately for untwisted and twisted components.

 Let $\Upsilon_j$ be an untwisted component $\Upsilon$, that is, $ i \neq \sigma (i) $ and $j= \{ i, \sigma (i) \}$.

(a) Let $v$ be a normal framing of $ \gamma_i \subset S^3 $, and let $ w$ be a normal framing of $ \gamma_{ \sigma (i)} \subset S^3 $. Then $df(v) + df(w)$ is not tangent to any of the two branches of $f$. Define 
\begin{equation*}
 c_j (v,w):= \lk_{ \R^5 } (f(S^3), \Upsilon_j + \delta (df(v) + df(w))) ,
 \end{equation*}
 where $ 0 < \delta $ is sufficiently small.

(b) Let $v$ be a normal framing of $f$ along $ \gamma_{ \sigma (i)} $, and let $ w$ be a nowhere zero normal vector field of $f$ along $ \gamma_{i} $. Define 
\begin{equation*}
c_j (v, w):=c_j( \Xi (v), \Xi (w)).
\end{equation*}

(c) Let $v$ be a normal framing of $ \gamma_i \subset S^3 $, and let $ w$ be a normal framing of $f$ along $ \gamma_{i} $. Define 
\begin{equation*}
c_j(v, w):=c_j (v, \Xi (w)).
\end{equation*}

\begin{remark}\label{re:lincomb}
 If $ v$ is a normal framing of $\gamma_i$ and $ w$ is a  normal framing of $ \gamma_{ \sigma (i)}$, and $ 0< \delta_1, \delta_{2} $ are sufficiently small real values (possibly depending on the points of $ \Upsilon$ continuously), then 
 \begin{equation*}
  \lk_{ \R^5 } (f(S^3), \Upsilon_j + \delta_1 df(v) + \delta_2  df(w))=c_j(v, w) ,
\end{equation*}
because $ \Upsilon_j + \delta_1 df(v) + \delta_2  df(w)) $ and $ \Upsilon_j + \delta (df(v) + df(w))$ are homotopic curves in $ \R^5 \setminus f(S^3)$.
\end{remark}
\begin{remark}\label{re:jajj} To avoid the complications described in Remark~\ref{re:fuas} we do not use the pushing out by normal framings of $f$. Instead we reduce those to the pushing out by a normal framing of $\gamma$ via $\Xi$.
	\end{remark}

To express the invariants $c_j$ in terms of $a$ and $b$ the key observation is the following.

\begin{lem}\label{le:kitolt}
	Let $ C \subset S^3 $ be an oriented link disjoint from $ \gamma $. Then 
 \begin{equation*}
 \lk_{ \R^5 } (f (S^3), f(C) + \delta \mathcal{N}|_C ) = \lk_{S^3} (C, \gamma) 
 \end{equation*}
 (where $ 0< \delta $ is sufficiently small).
\end{lem}

\begin{proof}
	Let $T \subset S^3$ be an oriented Seifert surface of $C$ whose intersection with $\gamma $ is transverse, and assume that if $p \in T \cap \gamma$, then $\iota(p) \notin T$. The shifted copy $T':=f(T)+\delta \mathcal{N}|_{T} $ is an oriented Seifert surface of $ f(C)+ \delta \mathcal{N}|_{C} $. For each $p \in T \cap \gamma$, the membrane $T'$ intersects $ f(S^3)$ at one point $q'$ near $q=f(p)=f( \iota(p))$ on the other branch of $f$ (the branch of $\iota (p)$), see Figure~\ref{fig:C}. The intersection point $ q' \in f(S^3) \cap T'$ has the same sign as $ p \in T \cap \gamma $, since the sum of the orientations of $ \Upsilon $ and the branches (in any order) agrees with the orientation of $ \R^5 $. Hence 
 \begin{equation*}
 \lk_{ \R^5 } (f (S^3), f(C) + \delta \mathcal{N}|_C )= \sharp_{\mbox{alg}} (f(S^3) \cap T')=\sharp_{\mbox{alg}}(T \cap \gamma)= \lk_{ S^3} (C, \gamma).
 \end{equation*}

\begin{figure}[h]

\centering
\resizebox{15cm}{6cm}{
\begin{picture}(0,0)%
\includegraphics{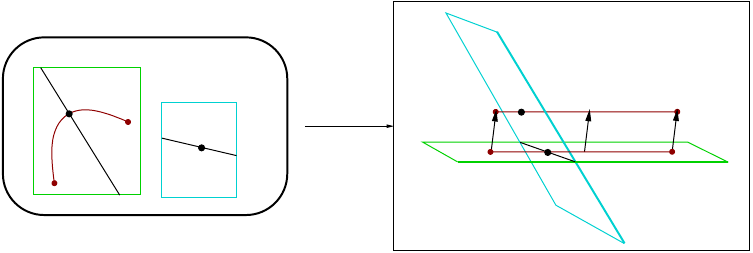}%
\end{picture}%
\setlength{\unitlength}{4144sp}%
\begin{picture}(5722,1920)(-1769,-481)
\put(-1211,1295){\makebox(0,0)[lb]{\smash{{\SetFigFont{12}{14.4}{\rmdefault}{\mddefault}{\updefault}{\color[rgb]{0,0,0}$S^3$}%
}}}}
\put(-350,378){\makebox(0,0)[lb]{\smash{{\SetFigFont{12}{14.4}{\rmdefault}{\mddefault}{\updefault}{\color[rgb]{0,0,0}$\iota(p)$}%
}}}}
\put( 57,195){\makebox(0,0)[lb]{\smash{{\SetFigFont{12}{14.4}{\rmdefault}{\mddefault}{\updefault}{\color[rgb]{0,0,0}$\gamma$}%
}}}}
\put(-1367,801){\makebox(0,0)[lb]{\smash{{\SetFigFont{12}{14.4}{\rmdefault}{\mddefault}{\updefault}{\color[rgb]{0,0,0}$\gamma$}%
}}}}
\put(-815,340){\makebox(0,0)[lb]{\smash{{\SetFigFont{10}{12.0}{\rmdefault}{\mddefault}{\updefault}{\color[rgb]{0,0,0}$C$}%
}}}}
\put(-1317,-16){\makebox(0,0)[lb]{\smash{{\SetFigFont{10}{12.0}{\rmdefault}{\mddefault}{\updefault}{\color[rgb]{0,0,0}$C$}%
}}}}
\put(-1438,526){\makebox(0,0)[lb]{\smash{{\SetFigFont{12}{14.4}{\rmdefault}{\mddefault}{\updefault}{\color[rgb]{0,0,0}$p$}%
}}}}
\put(-1025,606){\makebox(0,0)[lb]{\smash{{\SetFigFont{10}{12.0}{\rmdefault}{\mddefault}{\updefault}{\color[rgb]{0,0,0}$T$}%
}}}}
\put(2357,321){\makebox(0,0)[lb]{\smash{{\SetFigFont{8}{9.6}{\rmdefault}{\mddefault}{\updefault}{\color[rgb]{0,0,0}$q$}%
}}}}
\put(2730,371){\makebox(0,0)[lb]{\smash{{\SetFigFont{8}{9.6}{\rmdefault}{\mddefault}{\updefault}{\color[rgb]{0,0,0}$N$}%
}}}}
\put(3364, 69){\makebox(0,0)[b]{\smash{{\SetFigFont{8}{9.6}{\rmdefault}{\mddefault}{\updefault}{\color[rgb]{0,0,0}$f(C)$}%
}}}}
\put(1970, 95){\makebox(0,0)[b]{\smash{{\SetFigFont{8}{9.6}{\rmdefault}{\mddefault}{\updefault}{\color[rgb]{0,0,0}$f(C)$}%
}}}}
\put(2912,646){\makebox(0,0)[b]{\smash{{\SetFigFont{8}{9.6}{\rmdefault}{\mddefault}{\updefault}{\color[rgb]{0,0,0}$T'$}%
}}}}
\put(2856,126){\makebox(0,0)[b]{\smash{{\SetFigFont{8}{9.6}{\rmdefault}{\mddefault}{\updefault}{\color[rgb]{0,0,0}$f(T)$}%
}}}}
\put(2169,652){\makebox(0,0)[lb]{\smash{{\SetFigFont{8}{9.6}{\rmdefault}{\mddefault}{\updefault}{\color[rgb]{0,0,0}$q'$}%
}}}}
\put(3393,1104){\makebox(0,0)[lb]{\smash{{\SetFigFont{14}{16.8}{\rmdefault}{\mddefault}{\updefault}{\color[rgb]{0,0,0}$\mathbb R^5$}%
}}}}
\end{picture}%
} 
\caption{Proof of Lemma \ref{le:kitolt}}
\label{fig:C}
\end{figure}

\end{proof}

 \begin{remark}
Since the intersection of the branches of $f$ is not necessarily orthogonal (recall Remark~\ref{re:fuas}), $q'$ is not the shifted copy of $q$ by $\mathcal{N}$ -- see Figure~\ref{fig:C}.
In addition, the normal vectors $\mathcal{N}(p)$ are not concrete, since they are defined as elements of a quotient space, moreover, the framing $\mathcal{N}$ is defined only up to homotopy. Of course the argument of Lemma~\ref{le:kitolt} holds for every choice of representatives as long as $\delta$ is much smaller than the distance of $C$ and $\gamma$. 
\end{remark}

\begin{rem}
In the proof of Lemma~\ref{le:kitolt} we used only one property of $\mathcal{N}$, namely, that it extends from $C$ to a Seifert surface. If $C$ has only one component, then the converse is also true. Namely, if a nowhere zero section $v$ of $\nu(f)|_C$ extends to a nowhere zero section of $\nu(f)$ over a Seifert surface $T$ of $C$, then it extends to a global framing, that is, $v=\mathcal{N}|_C$. To show this, let $v$ and $w$ be two nowhere zero sections of the trivial bundle $\nu(f)|_T$. Their relative position defines a map $\phi: T \to SO(2)$. Its restriction $\phi|_C$ is null-homologous, hence it is nullhomotopic. Therefore if a framing of $f$ along $C$ extends to $T$, then it is homotopic to $\mathcal{N}|_C$ (note that it may still have non-homotopic extensions to $T$).
\end{rem}

\begin{cor}\label{co:fels}
(a)	 $ c_j (v, \mathcal{N}_{ i}) = a (v) $.

(b)	$ c_j (\mathcal{S}_i , \mathcal{N}_i)= 0 $.

(c) $	c_j(\mathcal{N}_{ \sigma (i)} , \mathcal{N}_i) = a (\mathcal{N}_{ \sigma (i)})= a (\mathcal{N}_i) $.
\end{cor}

\begin{proof} (b) and (c) follow from (a) and Proposition~\ref{pr:of}. To prove (a), let $v$ be a normal framing of $\gamma_i \subset S^3$. Choose sufficiently small positive numbers $ 0< \delta_2 \ll \delta_1 $. By Lemma \ref{le:kitolt}
\begin{equation*}
\lk_{\R^5}(f(S^3), f(\gamma_i + \delta_1 v)+ \delta_2 \mathcal{N}_i)=
\lk_{S^3}(\gamma_i + \delta_1 v, \gamma)
= a(v).
\end{equation*}
To show that the left hand side of the equation is equal to $ c_j (v, \mathcal{N}_{ i}) $ we have to manage the complications described in Remarks \ref{re:fuas} and \ref{re:jajj}.
According to the decomposition \eqref{eq:felb},
\begin{equation*}
\mathcal{N}_i = \mathcal{N}_i^{(\gamma_i)} + \mathcal{N}_i^{(\gamma_{\sigma(i)})} + \mathcal{N}_i^{(\Upsilon_j)} 
\end{equation*}
where $ \mathcal{N}_i^{(\gamma_{\sigma(i)})}= df( \Xi(\mathcal{N}_i)) $; $\mathcal{N}^{(\Upsilon_j)}_i $ is a tangent vector field of $ \Upsilon_j $; and $ \mathcal{N}_i^{(\gamma_i)} = df(w) $ holds with some $ w$ (possibly somewhere zero) normal vector field of $ \gamma_i \subset S^3 $. Then the curves
\[ f(\gamma_i + \delta_1 v)+ \delta_2 \mathcal{N}_i \approx f(\gamma_i + \delta_1 v+ \delta_2 w)+ df(\delta_2 \Xi(\mathcal{N}_i) )
\approx \Upsilon_j+  \delta_1 df( v)+\delta_2 df( \Xi(\mathcal{N}_i)) \] 
are homotopic in $ \R^5 \setminus f(S^3) $, since $ \gamma_i + \delta_1 v+ \delta_2 w $ and $ \gamma_i + \delta_1 v $ are homotopic in $ S^3 \setminus \gamma $. The linking number of 
\begin{equation*}
\Upsilon_j+  \delta_1 df( v)+\delta_2 df( \Xi(\mathcal{N}_i)) \mbox{ and } f(S^3)
\end{equation*}
in $\R^5$ is equal to $ c_j (v, \mathcal{N}_{ i}) $ by Remark~\ref{re:lincomb}.

\end{proof}

\begin{cor}
$ c_i (v^+, \mathcal{N}_{ i})	= c_i (v, \mathcal{N}_{ i}) +1 $.
\end{cor}

As another  consequence of Corollary~\ref{co:fels}, we can extend Corollary~\ref{co:diff1}.

\begin{cor}\label{co:diff2} The difference $a-b$ is an invariant of $f$ corresponding to the component $\Upsilon_j$ of $\Upsilon$. We denote it by $\Delta_j(f)$, and we have
\begin{equation*}
\Delta_j(f)=a(v)-b(v)=a(\mathcal{N}_i)
=
a(\mathcal{N}_{\sigma(i)})=
-b(\mathcal{S}_i)
=-b(\mathcal{S}_{\sigma(i)})=c_j(\mathcal{N}_{ \sigma (i)} , \mathcal{N}_i).
\end{equation*}
\end{cor}

In other words, the twisting of the global normal framing of $f$ restricted to $\gamma_{\sigma(i)}$,  considered as a normal framing of $\gamma_i$, is the same as the twisting of the global normal framing of $f$ restricted to $\gamma_{i}$,  considered as a normal framing of $\gamma_{\sigma(i)}$. This invariant $\Delta_j(f)$ also agrees with $c_j(\mathcal{N}_{ \sigma (i)} , \mathcal{N}_i)$ by (c) of Corollary~\ref{co:fels}. To obtain the complete picture, one more observation is required.

\begin{prop}\label{pr:step}
	$ c_j (v^+ , w)=c_j (v , w^+)= c_j (v, w)+1 $.
	\end{prop}

\begin{proof}
Recall that $v^+$ can be obtained by a canonical local modification of  $v$ over a small segment of $\gamma_i$, see Figure~\ref{fig:vplus}. $v$ and $v^+$ can be connected by a disc $D$ (with corners) having a $+1$ transverse intersection point $p$ with $\gamma_i$. By Remark~\ref{re:loc}, locally the framings $v$ and $w$ can be trivialized together, therefore $w$ can be extended to $D$. The curves
\begin{equation*}
\Upsilon_j+\delta (df(v)+df(w)) \mbox{ and } \Upsilon_j+\delta (df(v^+)+df(w)) 
\end{equation*}
can be connected by the shifted copy
\begin{equation*}
\widetilde{D}=f(D)+\delta df( w)
\end{equation*}
of $f(D)$ in the direction of $w$. The shifted disc $\widetilde{D}$ intersects $f(S^3)$ at one point $q'$ near $q=f(p)$. The intersection is transverse and by the compatible choice of orientations it witnesses an increase in the linking number by $+1$. 
\end{proof}

\begin{cor}\label{co:ckif}
\begin{equation*}
c_j (v,w)= a(v)+a(w)-a(\mathcal{N}_i)=a(v)+a(w)-\Delta_j(f),
\end{equation*}
in particular,
\begin{equation*}
c_j (\mathcal{S}_i , \mathcal{S}_{ \sigma (i)}) = - a(\mathcal{N}_i)=-\Delta_j(f) 
\end{equation*}
and 
\begin{equation*}
c_j (\mathcal{N}_i , \mathcal{N}_{ \sigma (i)}) = a(\mathcal{N}_i)=\Delta_j(f). 
\end{equation*}
\end{cor}

\begin{proof}
By Corollary~\ref{co:fels} we know that $c_j(v, \mathcal{N}_i)=a(v)$. Applying Proposition~\ref{pr:step}, $\mathcal{N}_i$ can be changed to $w$ step by step, increasing $c_j$ by $a(w)-a(\mathcal{N}_i)$. 
\end{proof}

\subsection{Pushing out of twisted components}\label{ss:twistedpush} Let $\Upsilon_j$ be a twisted component of $\Upsilon$, that is, $ i = \sigma (i) $ and $j= \{ i \}$.

In this case one vector field $v$, either a normal framing of $f$ along $ \gamma_i$ or a normal framing of $ \gamma_i \subset S^3 $, determines two vectors $ v(p_1) $ and $v(p_2)$ at $ q=f(p_1)=f(p_2) \in \Upsilon_j$. According to this we define the pushing out and the corresponding linking number in these two settings as follows.

\begin{enumerate}[(a)]
\item Let $v$ be a normal framing of $ \gamma_i \subset S^3$. Since $df_{p_1}(v(p_1))+df_{p_2}(v(p_2))$ is not tangent to any of the branches, we can define
\csere{\begin{equation}\label{eq:cjTwisted}
c_j (v):= \lk_{ \R^5 } (f(S^3), \widetilde{\Upsilon}_j^{(v)}),
\end{equation}}
where $\widetilde{\Upsilon}_j^{(v)}$ is the copy of $ \Upsilon_j$ shifted slightly along the vector field 
\begin{equation}\label{eq:cv}
q \mapsto df_{p_1}(v(p_1))+df_{p_2}(v(p_2)).
\end{equation}

\item Let $v$ be a normal framing of $f$ along $ \gamma_i $. Define \begin{equation*}
c_j(v):=c_j(\Xi(v)) 
\end{equation*}
\end{enumerate}

For a pair of framings we define the \emph{double pushing out} and the associated linking number, again depending on the kind of vector fields involved.

\begin{enumerate}[(a)]
\item Let $v$ and $w$ be normal framings of $ \gamma_i \subset S^3 $. Let $ \overline{\Upsilon}^{(v, w)} $ be the curve parametrized with $ \gamma_i $ by the map 
\begin{equation}\label{eq:dvw}
 p \mapsto f(p)+ \delta \cdot (df_p(v(p))+df_p(w(\iota(p))))
 \end{equation}
(where $ \delta $ is a sufficiently small positive real number). Then define 
\begin{equation*}
 d_j(v, w):= \lk_{\R^5} (f(S^3), \overline{\Upsilon}_j^{(v, w)}) .
 \end{equation*}

\item Let $v$ and $w$ be normal framings of $f$ along $ \gamma_i $. Define 
\begin{equation*}
d_j(v,w):=d_j (\Xi(v), \Xi(w)).
\end{equation*}

\item Let $v$ be a normal framing of $ \gamma_i \subset S^3 $ and let $w$ be a normal framing of $f$ along $ \gamma_i $. Define 
\begin{equation*}
 d_j(v, w) :=d_j(v, \Xi(w)).
\end{equation*}
\end{enumerate}

\begin{rem}\label{re:lincomb2} Similarly to Remark~\ref{re:lincomb}, the sum of the two vectors can be replaced by another positive linear combination to obtain the same linking numbers. Namely, in the definition of $c_j(v)$ the vector field \eqref{eq:cv} can be replaced by
\begin{equation*}
q \mapsto \delta_1 df_{p_1}(v(p_1))+ \delta_2 df_{p_2}(v(p_2)),
\end{equation*}
and in the definition of $d_j(v,w)$ the curve \eqref{eq:dvw} can be replaced by
\begin{equation*}
 p \mapsto f(p)+ \delta_1 df_p(v(p))+ \delta_2 df_p(w(\iota(p)))
\end{equation*}
with sufficiently small positive coefficients $\delta_1, \delta_2 > 0$ (possibly depending continuously on the point $q$ or $p$, respectively).

\end{rem}

\begin{prop}\label{prop:double}
$d_j(v, v)= 2 c_j(v) $.
\end{prop}

\begin{proof}
It follows directly from the constructions.
\end{proof}

Lemma \ref{le:kitolt} implies the following corollary, in analogy of Corollary~\ref{co:fels}.

\begin{cor}\label{co:fels2}
\begin{enumerate}[(a)]
\item $d_j (v, \mathcal{N}_i)=a(v) $.
\item $ d_j (\mathcal{S}_i, \mathcal{N}_i)=0 $.	
\item $ c_j(\mathcal{N}_i)= \frac{1}{2} a(\mathcal{N}_i) $.
\end{enumerate}
\end{cor}

\begin{proof}
Parts $(b)$ and $(c)$ follow immediately from $(a)$ using Proposition~\ref{pr:of} $(b)$ and Proposition~\ref{prop:double}, respectively.

Let $v$ be a nonzero normal field of $\gamma_i \subset S^3$. Choose sufficiently small positive numbers $ 0< \delta_2 \ll \delta_1 $. By Lemma \ref{le:kitolt}
\begin{equation*}
\lk_{\R^5}(f(S^3), f(\gamma_i + \delta_1 v)+ \delta_2 \mathcal{N}_i)= a(v). 
\end{equation*}
According to the decomposition \eqref{eq:felb} 
\begin{equation*}
  \mathcal{N}_i = \mathcal{N}_i^{(\gamma_i)} + \mathcal{N}_i^{(\gamma_{\sigma(i)})} + \mathcal{N}_i^{(\Upsilon_j)} ,
  \end{equation*}
  where 
  $  \mathcal{N}_i^{(\gamma_{\sigma(i)})}= df( \Xi(\mathcal{N}_i)) 
  $; $\mathcal{N}^{(\Upsilon_j)}_i $ is tangent to $ \Upsilon_j $; and $ \mathcal{N}_i^{(\gamma_i)} = df(w) $ with some $ w$ (possibly zero) normal vector field of $ \gamma_i \subset S^3 $. Then the curves
	\[ f(\gamma_i + \delta_1 v)+ \delta_2 \mathcal{N}_i \approx f(\gamma_i + \delta_1 v+ \delta_2 w)+ df(\delta_2 \Xi(\mathcal{N}_i) ) \]
 \[
	\approx [ p \mapsto f(p)+\delta_1 df(v(p))+ \delta_2df (\Xi(\mathcal{N}_i(p)))] \] 
	are homotopic in $ \R^5 \setminus f(S^3) $, since $ \gamma_i + \delta_1 v+ \delta_2 w $ and $ \gamma_i + \delta_1 v $ are homotopic in $ S^3 \setminus \gamma $. The linking number of the curve 
 \begin{equation*}
  [ p \mapsto f(p)+\delta_1 df(v(p))+ \delta_2df (\Xi(\mathcal{N}_i(p)))] 
  \end{equation*}
(where $p \in \gamma_i$) and $f(S^3)$  in $\R^5$ is equal to $ d_j (v, \mathcal{N}_{ i}) $ by Remark~\ref{re:lincomb2}.
\end{proof}

\begin{cor}
	$d_j (v^+, \mathcal{N}_i)=d_j (v, \mathcal{N}_i)+1 $.
\end{cor}

By Corollary~\ref{co:diff1}, $a(v)-b(v)=a(\mathcal{N}_i)=-b(\mathcal{S}_i)$, and in the twisted case it is trivially equal to $a(\mathcal{N}_{\sigma(i)})$ and $-b(\mathcal{S}_{\sigma(i)})$, hence it is an  invariant of $f$ corresponding to  $\Upsilon_j$. We use the notation $\Delta_j(f)=a(v)-b(v)=a(\mathcal{N}_i)=-b(\mathcal{S}_i)$  introduced in Corollary~\ref{co:diff2}.

\begin{prop}\label{pr:step2}
	$ d_j (v^+ , w)=d_j (v , w^+)= d_j (v, w)+1 $.
\end{prop}

\begin{proof}
Due to the local nature of the modification $v \mapsto v^+$, the proof of Proposition~\ref{pr:step} applies without change.
\end{proof}

\begin{cor}\label{co:dkif}
\begin{equation*}
d_j (v,w)= a(v)+a(w)-\Delta_j(f),
\end{equation*}
therefore
\begin{equation*}
c_j(v)=a(v)-\frac{1}{2} a(\mathcal{N}_i)=a(v)-\frac{1}{2} \Delta_j(f),
\end{equation*}
in particular,
\begin{equation*}
c_j (\mathcal{S}_i) =-\frac{1}{2} a(\mathcal{N}_i)= -\frac{1}{2} \Delta_j(f) ,
\end{equation*}
and 
\begin{equation*}
c_j (\mathcal{N}_i ) =\frac{1}{2} a(\mathcal{N}_i)= \frac{1}{2} \Delta_j(f). 
\end{equation*}
\end{cor}

\begin{proof}
By Corollary~\ref{co:fels2} we know that $d_j(v, \mathcal{N}_i)=a(v)$. Applying Proposition~\ref{pr:step2}, $\mathcal{N}_i$ can be changed to $w$ step by step, increasing $d_j$ by $a(w)-a(\mathcal{N}_i)$. 
\end{proof}

\subsection{\csere{Comparing the different versions of $L$}}

\csere{Recall the definition of $L_1(f)$ and $L_2(f)$ from Subsection~\ref{ss:eszinv}. Both are  the linking number of the image of $f$ with a shifted copy of the double point curve $\gamma$, but in the case of $L_1$ the Seifert framing $\mathcal{S}$ of $\gamma$ is used for the pushing out, while the global normal framing $\mathcal{N}$ of $f$ is used in the case of $L_2$. In terms of our newly defined invariants, $L_1(f)$ is the sum of $c_j(\mathcal{S}_i, \mathcal{S}_{\sigma(i)})=-a(\mathcal{N}_i)=-\Delta_j(f)$ for untwisted and $c_j(\mathcal{S}_i)=-\frac{1}{2}a(\mathcal{N}_i)=-\frac{1}{2}\Delta_j(f)$ for twisted components, while $L_2(f)$ is the sum of $c_j(\mathcal{N}_i, \mathcal{N}_{\sigma(i)})=a(\mathcal{N}_i)=\Delta_j(f)$ for untwisted and $c_j(\mathcal{N}_i)=\frac{1}{2}a(\mathcal{N}_i)=\frac{1}{2}\Delta_j(f)$ for twisted components, see Corollary~\ref{co:ckif}  and  Corollary~\ref{co:dkif}. Hence, we obtain the following.}

\begin{cor}\label{co:L12}
	\[ L_1 (f) = \sum_{  j \mbox{ untwisted} } c_j (\mathcal{S}_i, \mathcal{S}_{ \sigma (i)}) +
	\sum_{  j \mbox{ twisted} } c_j (\mathcal{S}_i) = -\frac{1}{2}  \sum_{i=1}^l a(\mathcal{N}_i) ,\]

	\[ L_2 (f) = \sum_{  j \mbox{ untwisted} } c_j (\mathcal{N}_i, \mathcal{N}_{ \sigma (i)}) +
	\sum_{  j \mbox{ twisted} } c_j (\mathcal{N}_i) = \frac{1}{2}  \sum_{i=1}^l a(\mathcal{N}_i) .\]
\end{cor}

\begin{cor}\label{co:L1L2equal}
	$ L_1 (f)= - L_2 (f) $.
\end{cor}

\begin{rem}\label{re:deltas}
$L_1$ and $L_2$ can be expressed by the invariants $\Delta_j(f)$ as
\begin{equation*}
L_1(f)=-\frac{1}{2} \sum_{i=1}^l \Delta_{\{i, \sigma(i)\}} (f) \mbox{ and } L_2(f)=\frac{1}{2} \sum_{i=1}^l \Delta_{\{i, \sigma(i)\}} (f).
\end{equation*}
This can be interpreted as 
\begin{equation}\label{eq:transfer}
L_1(f)=-\frac{1}{2} a(\mathcal{N}|_{\gamma})=\frac{1}{2}b(\mathcal{S}) \mbox{ and } L_2(f)=\frac{1}{2} a(\mathcal{N}|_{\gamma})=- \frac{1}{2} b(\mathcal{S}).
\end{equation}

To describe $a(\mathcal{N}|_{\gamma})$: the global normal framing  $\mathcal{N}|_{\gamma}$ of $f$ along $\gamma$ is transferred to a framing of $\gamma$ by $\Xi$, and $a(\mathcal{N}|_{\gamma})$ is its linking number with $\gamma$. Similarly, the Seifert framing $\mathcal{S}$ of $\gamma$ is transferred to a framing of $f$ along $\gamma$, and $b(\mathcal{S})$ is its linking number with the zero section. To prove equation~\eqref{eq:transfer} without componentwise decomposition, one can try to apply Lemma~\ref{le:kitolt} directly, in analogy of Corollaries~\ref{co:fels} and \ref{co:fels2}. However the technical difficulties (such as the mixed pushing out of twisted components by $\mathcal{S}$ and $\mathcal{N}|_{\gamma}$) naturally lead to our arguments in this section.
\end{rem}

\begin{rem}\label{re:Lv} To complete the picture, we express the version $L_v(f)$ of $L$ in terms of $a$, $b$ and $c$. The original definition is equation~\eqref{eq:Lv}.

Let $v$ be either a normal framing of $\gamma$ or the a normal framing of $f$ along $\gamma$. Let $v_i=v|_{\gamma_i}$ denote its restrictions to the components of $\gamma$. Then 
\begin{align*}
\widetilde{L}_v(f)= \sum_{j=\{i, \sigma(i)\} \in J} c_j(v_i, v_{\sigma(i)}) 
&=\sum_{i=1}^l (a(v_i)- \frac{1}{2}a(\mathcal{N}_i))=\sum_{i=1}^l a(v_i)+ L_1(f) \\
&= 
\sum_{i=1}^l (b(v_i)+ \frac{1}{2}a(\mathcal{N}_i))=
\sum_{i=1}^l b(v_i)+ L_2(f).
\end{align*}
cf. Corollaries \ref{co:diff2}, \ref{co:ckif} and \ref{co:dkif}.
Here $\sum_{i=1}^l b(v_i)$ is equal to the homology class of $[v] \in H_1(E_0(\nu(f)), \Z) \cong \Z$, considering $v$ a normal framing of $f$ along $\gamma$. This observation provides a new proof for \cite[Lemma 4.15.]{ekholm4}. Namely, 
\begin{equation*}
    L_v(f):=\widetilde{L}_v(f)-[v]=L_2(f),
\end{equation*} 
independently of the choice of $v$.
\end{rem}

\section{Nearby embeddings}\label{s:nearby}

\subsection{Preliminary summary of the section} For \csere{the image $f(S^3)$ of} a stable immersion $f: S^3 \looparrowright \R^5$ we define a family of nearby 3-manifolds in two versions. First\csere{, in Subsection~\ref{ss:abstnear}} we define them as abstract 3-manifolds by a surgery of $S^3$ along the set of double points $\gamma$, in analogy to formula~\eqref{eq:rag}. Then \csere{in Subsection~\ref{ss:embnear}} we realize them as embedded manifolds in $\R^5$ close to the image of $f$. In the case of immersions associated to finitely determined holomorphic germs the nearby 3-manifolds serve as topological candidates for the boundary of the Milnor fibre.

Every nearby 3-manifold is obtained by a resolution of the image of $f$ along each double value component $\Upsilon_j$ corresponding to a pair of normal framings $v_i$ and $v_{\sigma(i)}$, the restrictions of a framing $v$ of $\gamma$. We show that the resulting 3-manifold depends only on $a(v_i)+a(v_{\sigma(i)})$, or, equivalently, on $c_j(v_i, v_{\sigma(i)})$\csere{, see Proposition~\ref{pr:gluingdepend}}. \csere{By the main observation of the section, Proposition~\ref{pr:hurkkey}, the following two linking numbers (in $\R^5$) are equal:
\begin{itemize}
    \item The linking number of the nearby 3-manifold of $f(S^3)$ corresponding to the pair of framings $(v_i, v_{\sigma(i)})$ and $\Upsilon_j$,
    \item $c_j(v_i, v_{\sigma(i)})$,  which is the linking number of $f(S^3)$  and the push-out copy of $\Upsilon_j$ corresponding to the pair of framings $(v_i, v_{\sigma(i)})$.
\end{itemize}
(Here in the summary we formulated these statements for untwisted components, but each has a counterpart for twisted components, detailed in the section proper.)} Based on this property, we are able to describe the pairs of framings which provide the boundary of the Milnor fibre as the corresponding nearby 3-manifold\csere{, see Section~\ref{s:topref}}.

\subsection{Abstract nearby 3-manifolds}\label{ss:abstnear} 
Let $v$ be a normal framing of $ \gamma \subset S^3 $. According to the notation \csere{used in Subsection~\ref{ss:framlink} and} Remark~\ref{re:Lv}, let $v_i=v|_{\gamma_i}$ denote the restrictions of $v$ to the components of $\gamma$. \csere{Recall that a normal framing of $f$ along $\gamma$ (like the restriction of the global normal framing $\mathcal{N}$) also induce a normal framing of $\gamma \subset S^3$ via the isomorphism described in Proposition~\ref{pr:isom}.}

We associate a $3$-manifold constructed by the following surgery. 
First we cut out some open tubular neighbourhoods of all components $ \gamma_i \subset S^3 $, then we glue the torus boundaries as in the construction of the boundary of the Milnor fibre, cf. \eqref{eq:rag}, but we use different gluing maps, determined by the framing $v$. Hence we define the abstract nearby manifold $ M[v] $ associated to the framing $ v$ as

\begin{equation}\label{eq:ragtop}  M[v] \simeq \left[ \left(S^3 \setminus
\bigcup_{i=1}^l {\rm int}(N_i)  \right)\cup \left(\bigcup_{i = \sigma (i) } Y_i \right) \right] 
\Bigg / \phi \mbox{,} \end{equation}
where  $N_i$ is a sufficiently small closed tubular neighbourhood of $ \gamma_i \subset S^3$, each $ Y_i $ is diffeomorphic to $Y$ \csere{(defined by equation~\eqref{eq:3Y})}, and the gluing $ \phi $ is induced by a
collection $( \phi_i)_{i=1, \dots, l}$ of diffeomorphisms
\[ \phi_{i}: \left\{ \begin{array}{ccc} \partial N_i \to -\partial N_{\sigma(i)}
& \mbox{ if } & i \neq \sigma(i) \\
\partial N_i \to \partial Y_i \hspace{1cm} & \mbox{ if } & i = \sigma(i). \\
\end{array} \right.    \]

The gluing functions $ \phi_i $ are determined up to isotopy by their induced homomorphism on the first homology. We choose a basis $(m_i, l_i) $ for each $H_1(\partial N_i, \Z) \cong \Z \oplus \Z$. Let  $m_i$ be the oriented meridian of $ \partial N_i $. The framing $v_i$ defines an oriented longitude $l_i=\gamma + \delta v_i$ with a sufficiently small $\delta > 0$. Recall that an oriented meridian $m$ and an oriented longitude $c$ on $\partial Y$ are fixed, see Preliminaries (Subsection~\ref{ss:boundmilnor}) or \cite[3.2.]{NP2}. Then the gluing maps are defined as follows:

\[ \begin{array}{cccc} \mbox{Untwisted case } (i \neq \sigma(i)) \mbox{:} & 
m_i \mapsto -m_{ \sigma(i)} & \mbox{ and } &
 l_i \mapsto l_{ \sigma(i)}, \\  
 \mbox{Twisted case } (i = \sigma(i)) \mbox{:} &
 m_i \mapsto -m & \mbox{ and } &
 l_i \mapsto c. \end{array} \]

\begin{prop}\label{pr:gluingdepend}
\phantom{Text to not break linking.}
\begin{enumerate}[(a)]
\item For untwisted components, \csere{two} framings $v$ and $w$ of $\gamma$ induce the same gluing map $ \phi_i :\partial N_i \to \partial N_{\sigma(i)}$ if and only if 
\begin{equation*} 
a(v_i)+a(v_{\sigma(i)}) = a(w_i)+a(w_{\sigma(i)}) ,
\end{equation*}
or, equivalently, $ c_j (v_i, v_{\sigma(i)})= c_j (w_i, w_{\sigma(i)})$.

\item \csere{For twisted components, framings $v$ and $w$ with $a(v_i) \neq a(w_i)$ induce different gluing maps $\phi_i$.}
\end{enumerate}
\end{prop}

\begin{proof} 
\begin{enumerate}[(a)]
\item Consider the image of the homology class corresponding to $v_i^+ $ by $\phi_i$: 
\[ l_i^+=l_i+m_i \mapsto l_{\sigma(i)}-m_{ \sigma(i)} , \]
and the right hand side is $l_{\sigma(i)}^{-}$ in the sense that its linking number with $\gamma$ is $a(v_{\sigma(i)})-1$. 
Hence another choice $w_i$, $w_{ \sigma(i) }$ induces the same gluing map $ \phi_i $ if and only if $ a(v_i)+a(v_{\sigma(i)}) = a(w_i)+a(w_{\sigma(i)}) $.

\item \csere{If $l_i$ is mapped to $c \subset \partial Y$ by $\phi_i$, then $l_i^+=l_i+m_i$ is mapped to $c-m$; iterating this to get from $v_i$ to $w_i$ we can never get back to $c$.}
\end{enumerate}
\end{proof}

\subsection{Embedded nearby 3-manifolds}\label{ss:embnear} We construct an embedding $ M[v] \subset \R^5 $ by realizing the surgery in a small neighbourhood of each $ \Upsilon_j \subset \R^5 $. First we define the gluing pieces, then we describe the identification.

By Remark~\ref{re:loc}, every point $f(p)=f(\iota(p))=q \in \Upsilon_j$ admits a $\mathcal{C}^{\infty}$ local chart $ (x, y, \tau) \in \C \times \C \times I $ on a neighborhood $U_q \subset \R^5$ such that $ f(S^3)= \{xy=0\}$,  $\Upsilon=\{(0,0,\tau)\}$ and
\begin{equation}\label{eq:loktriv2}
\begin{aligned}
\left\{f(\gamma_i +s v_i) \ | \ 0 < s < \delta \right\} &=\left\{(t, 0, \tau) \ | \ t \in \R^+\right\}, \\
\left\{ f(\gamma_{\sigma(i)}+s v_{\sigma(i)}) \ | 0 < s < \delta \right\} &=\left\{(0, t, \tau)  \ | \ t \in \R^+\right\}.
\end{aligned}
\end{equation}

Moreover the chart can be chosen such that the shifted copy of $\Upsilon_j $ is $(1, 1, \tau )$. Recall that $\Upsilon_j^{(v)}$  is the copy of $ \Upsilon_j $ shifted slightly along the vector field $ df(v_i)+df(v_{ \sigma(i)}) $ in the untwisted case, and along the vector field $ q \mapsto df_{p_1}(v_i(p_1))+df_{p_2}(v_i(p_2)) $ (where $q=f(p_1)=f(p_2)$) for twisted components.


For the purpose of handling the gluing map we need to extend this local description to an entire neighbourhood of $\Upsilon_j$. In the untwisted case the normal bundles of the preimages $\gamma_i$ and $\gamma_{\sigma(i)}$ are trivial (as oriented vector bundles over an $S^1$), and we can choose the trivialization in such a way that the vector fields $v_i$ and $v_{\sigma(i)}$ become constant sections. Hence in the untwisted case, we have a diffeomorphism between a neighbourhood $U_j$ of $\Upsilon_j$ and $\C \times \C \times S^1$ in which the description \eqref{eq:loktriv2} holds. For a twisted component, we can use the same argument to see that the normal bundle of $\gamma_i$ is trivial, and we can take a trivialization with $v_i$ being the constant $1 \in \C$ section. Pushing this trivialization forward to $\R^5$ we get an identification of a neighbourhood $U_j$ of $\Upsilon_j$ with the mapping torus of the linear map $(x,y) \mapsto (y,x)$, $x,y\in \C$, where the shifted copy of $\Upsilon_j$ is again the section $(1,1)$ (well-defined at the gluing fiber of the mapping torus as well).

Define
\[\ \widetilde{M}_{q}[v]=\{(x, y, \tau) \in U_{q} \ | \ xy=1\} \simeq S^1 \times I \times I .\]

\noindent These pieces together form a submanifold in $ \R^5$ :
\[\widetilde{M}_j[v]=\bigcup_{q \in \Upsilon_j} \widetilde{M}_{q}[v] \subset U_j  \]

\noindent $\widetilde{M}_j[v]$ is diffeomorphic to $S^1 \times I \times S^1 $ for untwisted components, and it is diffeomorphic to $Y$ for twisted components.
In both cases, $\Upsilon_j^{(v)}$ is a curve contained in $ \widetilde{M}_j[v]$.
 
Intuitively the boundaries $ \partial \widetilde{M}_j[v] $ and $ \partial (U_j \cap f(S^3))$ are close to each other. We can identify them by an isotopy near the boundaries in $ \partial U_j $, defined as follows. Pick a smooth function $\eta:[0,\infty) \to [0,1]$ such that $\eta = 1$ on $[0,2]$ and $\eta=0$ far from $0$. Define $\Psi: \C^2 \to \C^2$ as

\begin{equation*}
    \Psi(x, y)=\left(\eta(|y|) x, \eta(|x|)y\right).
\end{equation*}
$\Psi$ is the identity map in a neighborhood of the origin that contains the point $(1,1)$. Far from the origin small neighborhoods of the coordinate axes are collapsed onto the corresponding axes, that is, if $|x|$ is small and $|y|$ is large, then $\Psi(x, y)=(0, y)$, and if $|y|$ is small and $|x|$ is large, then $\Psi(x, y)=(x, 0)$. 

Consider the images of the above defined pieces,
\begin{equation*}
\Psi(\widetilde{M}_{p_0}[v])=\{(\eta(|y|)x, \eta(|x|)y, \tau) \in U_{p_0} \ | \ xy=1\}.
\end{equation*}
This modifies the cylinders only on the `edges' where either $|x|$ is large and $|y|$ is small or vice versa, so there our set stays an embedded manifold. The union of $\Psi(\widetilde{M}_{p_0}[v])$ for every point of $\Upsilon_j$ form the submanifold $\Psi(\widetilde{M}_j[v]) \subset U_j$ isotopic to $\widetilde{M}_j[v]$. Its boundary is exactly the same as $ \partial (U_j \cap f(S^3))$, and $\Upsilon_j^{(v)}$ is contained in $\Psi(\widetilde{M}_j[v])$.

Now we are able to define
\[ \widetilde{M}[v] := \left(f(S^3) \setminus \bigcup_{j \in J} U_j \right) \underset{\partial}{\cup} \left( \bigcup_{j \in J} \Psi(\widetilde{M}_j[v]) \right). \]
$\widetilde{M}[v]$ is a closed oriented submanifold of $ \R^5 $.

\begin{prop}
	$\widetilde{M}[v]$ is diffeomorphic to $ {M}[v]$.
\end{prop}

\begin{proof} Consider the union of the curves 
 $ \{(\alpha  , \frac{1}{\alpha})  \ | \ \alpha \in \R^+ \} )$ on the cylinders $xy = 1$ for all $\tau$. Its image under $\Psi$ joins the longitudes $l_i$ and $l_{\sigma(i)}$ in $\Psi(\widetilde{M}_j[v])$. Hence $\widetilde{M}[v]$ is constructed by the same surgery as $M[v]$, see Formula \eqref{eq:ragtop} and the gluing function $\phi_i$.
\end{proof}

\begin{prop}\label{pr:hurkkey} \csere{For untwisted components,}
	\[	\lk_{ \R^5} \left(\widetilde{M}[v] , \Upsilon_j\right)=c_j (v_i, v_{ \sigma (i)} ) \]
    \csere{and for twisted components
    \[	\lk_{ \R^5} \left(\widetilde{M}[v] , \Upsilon_j\right)=c_j (v_i ) \]
    hold.}
\end{prop}

\begin{proof}
\csere{Consider the untwisted case first.} We choose $0 \ll \delta_1 \ll \delta_2$ sufficiently small values, then define the shifted copy $\Upsilon^{(v)}_j$ using $\delta_2$ and the nearby manifold $\widetilde M[v]$ using $\delta_1$. Then for any $\lambda \in \C \setminus \{ 1, -1 \}$, $|\lambda|=1$ we have

\[
\begin{aligned}
c_j(v_i, v_{\sigma(i)}) &\overset{\text{\ding{172}}}{=} \lk\left(f(S^3),\Upsilon^{(v)}_j\right) \overset{\text{\ding{173}}}{=} \lk\left(\widetilde{M}[v],\Upsilon^{(v)}_j\right) \overset{\text{\ding{174}}}{=} \lk\left(\widetilde M[v],\Upsilon_j+\delta_2(v_i+v_{\sigma(i)})\right)= \\
&\overset{\text{\ding{175}}}{=} \lk\left(\widetilde M[v], \Upsilon_j+\lambda\cdot\delta_2(v_i+v_{\sigma(i)})\right) \overset{\text{\ding{176}}}{=} \lk\left(\widetilde M[v],\Upsilon_j\right)
\end{aligned}
\]
Equality \ding{172} is the definition of $c_j$. In equality \ding{173}, we are changing $f(S^3)$ to the homologous manifold $\widetilde M[v]$, using the surface $W=\{(x,y,\tau) \mid 0 \leq xy \leq 1 \}$; the linking number does not change because the curve $\Upsilon^{(v)}_j$ does not intersect $W$ at all. Equality \ding{174} is again just the definition of the shifted copy of $\Upsilon_j$. Equality \ding{175} holds because we can connect the two shifted copies of $\Upsilon_j$ by a homotopy that avoids $\widetilde M[v]$. Indeed, if $s:[0,1] \to \C$ parametrizes a unit circle arc from $s(0)=1$ to $s(1)=\lambda$, then $t \mapsto \Upsilon_j + s(t) \cdot \delta_2(v_i+v_{\sigma(i)})$ is such a homotopy. Finally, \ding{176} holds because the linear homotopy between $\Upsilon_j$ and its shift by $\lambda \cdot \delta_2(v_i+v_{\sigma(i)})$ contains only points whose local $x$ and $y$ coordinates are either both $0$ (so the point is on $\Upsilon_j$) or both have argument equal to that of $\lambda$, implying $xy \not \in \R$ and consequently avoiding the set $\{xy=1\}$.

\csere{For the twisted case, we have an analogous sequence of equalities:}
\[
\begin{aligned}
c_j(v_i) &=\lk\left(f(S^3),\widetilde\Upsilon^{(v)}_j\right) = \lk\left(\widetilde{M}[v],\widetilde\Upsilon^{(v)}_j\right) = \lk\left(\widetilde M[v],\Upsilon_j+\delta_2(v_i+v_i\circ \iota)\right)= \\
&= \lk\left(\widetilde M[v], \Upsilon_j+\lambda\cdot\delta_2(v_i+v_i \circ \iota)\right) = \lk\left(\widetilde M[v],\Upsilon_j\right)
\end{aligned}\]
\csere{with the first equality the defining equation~\eqref{eq:cjTwisted} of $c_j$ and the rest completely analogous to the untwisted case.}
\end{proof}



\section{Topological reformulation of the vertical indices}\label{s:topref}
\subsection*{}
\setcounter{subsection}{1}
Consider a finitely determined holomorphic germ $\Phi: (\C^2, 0) \to (\C^3, 0)$. Recall from the Preliminaries (Subsection~\ref{ss:boundmilnor}) that $\Phi|_{\mathfrak{S}} : \mathfrak{S} \simeq S^3 \looparrowright S^5$ denotes the stable immersion associated with $\Phi$. The image of $\Phi$ is  a non-isolated singularity $(X, 0) \subset (\C^3, 0)$ defined by the zero set of a germ of function $g: (\C^3, 0) \to (\C, 0)$. The Milnor fibre $F$ of $(X,0)$ is defined as the $\delta$-level set of $g$ intersected with the Milnor ball $B^6_{\epsilon}$.

 By the decomposition formula \eqref{eq:rag}, the boundary $\partial F$ is an embedded nearby $3$-manifold of the associated stable immersion $ \Phi|_{ \mathfrak{S}}$, that is,  $ \partial F \simeq \widetilde{M}[v] \subset S^5 $ with a suitable choice of normal framing $v$ of $ \gamma \subset \mathfrak{S} $.

\begin{prop}
	For all $j \in J $, $ \lk_{ S^5} (\partial F, \Upsilon_j)=0$.
\end{prop}

\begin{proof} $ \Upsilon_j = \partial \Sigma_j \subset f^{-1}(0) $ and $ F = f^{-1}( \delta) \cap B^6_{ \epsilon} $, hence $F \cap \Sigma_j = \emptyset$.
\end{proof}

\csere{Recall that by Proposition~\ref{pr:hurkkey}  the linking number of the nearby 3-manifold of $\Phi|_{\mathfrak{S}}(\mathfrak{S})$ corresponding to $v$ (which is $\partial F$ in this case) and $\Upsilon_j$ is equal to  $
	c_j(v_i,v_{\sigma(i)})$ ($c_j(v_i)$ in the twisted case). We conclude that}
\begin{cor}\label{co:cmiln}
	$
	c_j(v_i,v_{\sigma(i)})=0$ \csere{holds for untwisted ($i \neq \sigma(i) $), and $
	c_j(v_i)=0$ holds for twisted ($i = \sigma(i) $)components.}
\end{cor}

Comparing this with Corollary \ref{co:fels} \csere{stating that $ c_j (\mathcal{S}_i , \mathcal{N}_i)= 0 $ holds in the untwisted case,} that means $ \mathcal{S}_i $ and $ \mathcal{N}_i $ glue together in the construction of $ \partial F$. That is, with the choice $v_i=\mathcal{S}_i$ on $\gamma_i$, its pair is $v_{\sigma(i)}=\mathcal{N}_i$ on $\gamma_{\sigma(i)}$.  In the twisted case \csere{ Corollary \ref{co:dkif} states that $c_j(v_i)=a(v_i)-\frac{1}{2} a(\mathcal{N}_i)$, which together with $c_j(v_i)=0$ implies that} $a(v_i)=\frac{1}{2} a(\mathcal{N}_i)$ holds for the framing $v_i$ that glues to the longitude $c$ of $\partial Y$ \csere{. That is,} $v_i$ is the average between $\mathcal{S}_i$ and $\mathcal{N}_i$.

Recall that $C \cdot D$ denotes the intersection multiplicity of two complex plane curves $(C, 0), (D,0) \subset (\C^2, 0)$ at 0.

\begin{prop}\label{pr:toprefvert} The vertical index $\mathfrak{vi}_j$ of the component $j=\{i, \sigma(i)\}$ can be expressed as
\begin{equation}\label{eq:reform}\mathfrak{vi}_j=
\left\{	\begin{array}{ccc}
a(\mathcal{N}_i)-\sum_{k \neq i} D_i \cdot D_k - \sum_{k \neq \sigma(i)} D_{\sigma(i)} \cdot D_k & \mbox{if} & i \neq \sigma(i) \\
\frac{1}{2} a(\mathcal{N}_i)- \sum_{k \neq i} D_i \cdot D_k &  \mbox{if} & i = \sigma(i) ,\\
 \end{array}
\right.
\end{equation}
\end{prop}

\begin{proof} Although the description of the gluing of $\partial F$ is simple in terms of the Seifert framing $\mathcal{S}_i$ of $\gamma $ and the global normal framing $\mathcal{N}_i$ of the immersion $\Phi|_{\mathfrak{S}}$, the vertical indices describe the gluing in terms of the Seifert framing $s_i$ of $\gamma_i$. The intersection multiplicities in equation \eqref{eq:reform} express the difference between $\mathcal{S}_i$ and $s_i$, indeed, $D_i \cdot D_k$ is equal to the linking number of their links $\gamma_i$ and $\gamma_k$. In particular:

\emph{Untwisted case $ i \neq \sigma(i)$:} Choose $v_i = s_i $, then  $ \mathfrak{vi}_j = a_{ \sigma(i)} (v_{ \sigma(i)})$ holds for its pair $v_{\sigma(i)}$. \csere{Indeed, $s_i$ is the topological longitude (Seifert framing) of $\gamma_i$, and, on the one hand, it glues to the framing $v_{\sigma(i)}$ on $\gamma_{\sigma(i)}$, on the other hand, by Theorem~\ref{th:milnper}, $s_i$  glues to the topological longitude $s_{\sigma(i)} $ plus $\mathfrak{vi}_j$ times the oriented meridian. This means that the linking number of $v_{\sigma(i)}$ with $\gamma_{\sigma(i)}$ is equal to $\mathfrak{vi}_j$, but this linking number is equal to $a_{ \sigma(i)} (v_{ \sigma(i)})$ by definition, see equation~\eqref{eq:ai}.} By \csere{point (f) of} Proposition \ref{pr:of}, 
\csere{\[ a(s_i)= \sum_{k \neq i } \lk_{ S^3} ( \gamma_i, \gamma_k)= \sum_{k \neq i} D_i \cdot D_k . \]}
\csere{By Corollary~\ref{co:ckif} we have $c_j (v_i,v_{ \sigma(i)})= a(v_i)+a(v_{ \sigma(i)})-a(\mathcal{N}_i)$, and by Corollary~\ref{co:cmiln} we have $c_j (v_i,v_{ \sigma(i)})=0$, hence we obtain that}
 \[ a(v_{ \sigma(i)})=
 c_j(v_i, v_{\sigma(i)})+a(\mathcal{N}_i)-a(v_i)
 =a(\mathcal{N}_i)- \sum_{k \neq i} D_i \cdot D_k. \] 
\csere{Comparing it with point (d) of} Proposition~\ref{pr:of} \csere{then} implies
\[ \mathfrak{vi}_j = a_{ \sigma(i)} (v_{ \sigma(i)})=a (v_{ \sigma(i)})-\sum_{k \neq \sigma(i)} D_{\sigma(i)} \cdot D_k,\]
\csere{and the first case of} equation~\eqref{eq:reform} \csere{is proved}.

\emph{Twisted case $ i= \sigma(i)$:} Let $v_i$ be the framing of $ \gamma_i \subset \mathfrak{S} $ which glues to the longitude $c$ of $ \partial Y $ in the construction of $ \partial F $.
\csere{As discussed after Corollary~\ref{co:cmiln}, $a(v_i)= \frac{1}{2} a(\mathcal{N}_i) $ (this follows from Corollaries~\ref{co:cmiln} and \ref{co:dkif}). We also have  $ \mathfrak{vi}_j = a_i (v_i) $. Indeed, on one hand,  $c$ glues to the topological longitude $s_i$ plus $\mathfrak{vi}_j$ times the meridian by Theorem \ref{th:milnper}, on the other hand, $c$ glues to the framing $v_i$ on $\gamma_i$. This means that the linking number of $v_i$ with $\gamma_i$ is equal to  $\mathfrak{vi}_j$, but this linking number is equal to $a_i(v_i)$ by definition. By point (d) of Proposition~\ref{pr:of} we have
\[ \mathfrak{vi}_j = a_i (v_i)=a (v_i)-\sum_{k \neq i} D_i \cdot D_k,\]
and substituting $a(v_i)= \frac{1}{2} a(\mathcal{N}_i) $ proves the second case of equation~\eqref{eq:reform}.
}
\end{proof}

\csere{Recall that by our convention $L=L_1=-L_2$ (see Subsection~\ref{ss:eszinv}), and $L_1(\Phi|_{\mathfrak{S}})=-\frac{1}{2}  \sum_{i=1}^l a(\mathcal{N}_i)$ holds by Corollary~\ref{co:L12}. Therefore, for the sum of the right hand side of equation~\eqref{eq:reform} we obtain that} 

\begin{cor}\label{cor:sumofthe}
	\[\sum_{j \in J} \mathfrak{vi}_j =-  \sum_{i \neq k} D_i \cdot D_k -L(\Phi|_{\mathfrak{S}})  = -  \sum_{i \neq k} D_i \cdot D_k +3 T( \Phi)-C(\Phi) .
	\]
\end{cor}
\csere{The second equation follows from the equality $L(\Phi|_{\mathfrak{S}})=C(\Phi)-3T(\Phi)$ proved in \cite{NP, PS}.}
Although $L$ is a topological invariant which is usually hard to compute directly, the formula establishes a correspondence between invariants which can be computed in an algebraic way, see the examples below in Subsection~\ref{ss:examp}.

\begin{rem} These results show that the natural choice for the longitude of $ \partial N_i$ to describe the gluing of $ \partial F $ is $\mathcal{S}_i $, instead of ${s}_i$. If  in Theorem \ref{th:milnper} one uses the homological trivialization  $ (m_i, \mathcal{S}_i)$ of $ \partial N_i $, where $m_i$ is the oriented meridian, the gluing maps  read as  
	\begin{equation*} i \neq \sigma(i) \mbox{: \ \ }
\left( \begin{array}{cc}
-1 & a(\mathcal{N}_i) \\
0 & 1 \\
\end{array} \right)
\end{equation*}  
 \begin{equation*} i = \sigma(i) \mbox{: \ \ }
\left( \begin{array}{cc}
-1 & \frac{1}{2} a(\mathcal{N}_i) \\
0 & 1 \\
\end{array} \right)
\end{equation*}  

With this choice the sum of the gluing coefficients is \csere{$L_2(\Phi|_{ \mathfrak{S}})=-L(\Phi|_{ \mathfrak{S}})=3 T( \Phi)-C(\Phi) $}. 
\end{rem}

\begin{cor}\label{co:vitopinv}
    The set of the vertical indices $\{\mathfrak{vi}_j\}_{j \in J}$ is a topological invariant of $\Phi$, meaning that it is invariant under topological left-right equivalence. Therefore the diffeomorphism type of the boundary of the Milnor fibre is a topological invariant.
\end{cor}

\csere{\begin{rem}
Before the proof we notice the following.
\begin{itemize}
    \item Corollary~\ref{cor:sumofthe} implies that the \emph{sum} of the vertical indices is a topological invariant, but here we show a stronger statement, namely, the \emph{set} of the vertical indices is also topological.
    \item The proof of Corollary~\ref{co:vitopinv} does not use Corollary~\ref{cor:sumofthe}. Both Corollaries~\ref{cor:sumofthe} and \ref{co:vitopinv} are consequences of the topological reformulation of the vertical indices in Propostion~\ref{pr:toprefvert}.
    \item One can observe that  Corollaries ~\ref{cor:sumofthe} and \ref{co:vitopinv} (together with the knowledge that the intersection multiplicities of $D_i$ are topological) provide a new proof for the topological invariance of $L(\Phi|_{\mathfrak{S}})=C(\Phi)-3T(\Phi)$. However, the topological invariance of the vertical indices follows from our topological reformulation, which is similar to the construction of $L$. In fact, the proof of Corollary~\ref{co:vitopinv} is analogous to the proof of the topological invariance of $L(\Phi|_{\mathfrak{S}})$  in \cite[Proposition 3.2.2.]{PS}.
\end{itemize}
\end{rem}}

\begin{proof}
The topological left-right equivalence of two finitely determined germs $\Phi$ and $\Phi'$ means that there exist germs of  homeomorphisms $ \phi: (\C^2, 0) \to ( \C^2, 0)$ and $ \psi: (\C^3, 0) \to ( \C^3, 0)$ such that 
$\Phi'=\psi \circ \Phi \circ \phi$
holds. 
\begin{equation*} 
\begin{tikzcd}
( \C^2, 0) \arrow[r, "\Phi"] 
& ( \C^3, 0) \arrow[d, "\psi"] \\
( \C^2, 0) \arrow[r, "\Phi'"]
\arrow[u, "\phi"]
&  ( \C^3, 0)
\end{tikzcd}
\end{equation*}

The double point curves $(D, 0)=(d^{-1}(0), 0)$ of $\Phi$ and $(D', 0)=(d'^{-1}(0), 0)$ of $\Phi'$ are topologically equivalent germs of curves, in fact, $D=\phi(D')$. 
Their links $\gamma$, $\gamma'$ are of the same type as links in $\mathfrak{S}\simeq \mathfrak{S}'\simeq  S^3$. 

We have to show that the right hand side of equation~ \eqref{eq:reform} agrees for $\Phi$ and $\Phi'$. Since $D_i \cdot D_k$ is equal to the linking number of $\gamma_i$ and $\gamma_k$, these terms are equal. We only have to show that $a(\mathcal{N}_i)=a(\mathcal{N}'_i)$. $a(\mathcal{N}_i)$ is equal to $-c_j(\mathcal{S}_i, \mathcal{S}_{\sigma(i)})$ for untwisted components, and $-2c_j(\mathcal{S}_i)$ for twisted components.

$\mathcal{S}'_i$ can be represented by the normal framing $\mbox{grad} (d')$ along $\gamma'_i$. Although the normal framing $\mbox{grad} (d')$ cannot be pushed forward by $\phi$ since $\phi$ is not necessarily differentiable, the slightly shifted copy $\widetilde{\gamma}'$ of $\gamma'$ along $\mbox{grad} (d')$ can be. 
The image $\phi(\widetilde{\gamma}')$ determines a normal vector field denoted by $\phi_* (\mbox{grad} (d'))$ along $\gamma$, which is homotopic to $\mbox{grad}(d)$, both vector field represents the Seifert framing $\mathcal{S}$ along $\gamma$, and their restrictions to a component $\gamma_i$ represent $\mathcal{S}_i$.

The sum of the two copies of $d \Phi (\mbox{grad}(d))$ and $d \Phi (\phi_* (\mbox{grad} (d')))$ are homotopic normal fields along $\Upsilon_j$, thus the shifted copies $\widetilde{\Upsilon}_j$ and $\widetilde{\Upsilon}_j^{(2)}$ of $\Upsilon_j$ along these vector fields are homotopic in $S^5 \setminus \Phi(\mathfrak{S})=S^5 \setminus \Phi(\phi(\mathfrak{S}'))$. 

Therefore, $\lk_{S^5} (\Phi(\mathfrak{S}), \widetilde{\Upsilon}_j)=\lk_{S^5} (\Phi(\mathfrak{S}), \widetilde{\Upsilon}_j^{(2)})$. Applying $\psi$ to the whole configuration does not change the linking numbers, and $\psi (\Phi(\mathfrak{S}))=\Phi'( \mathfrak{S}')$, $\psi(\Upsilon)=\Upsilon'$, $\psi(\widetilde{\Upsilon}^{(2)})=\widetilde{\Upsilon}'$, therefore we obtain that
\[c_j(\mathcal{S}_i, \mathcal{S}_{\sigma(i)})=\lk_{S^5} (\Phi(\mathfrak{S}), \widetilde{\Upsilon}_j)=\lk_{S^5} (\Phi'(\mathfrak{S}'), \widetilde{\Upsilon}'_j)=c_j(\mathcal{S}'_i, \mathcal{S}'_{\sigma(i)}),\]
which proves the Corollary.
\end{proof}

\subsection{Corank--1 germs with \texorpdfstring{$T=0$}{T=0}}\label{ss:cor1T0}

 In \cite{NP2} equation~\eqref{eq:osszeg} is proved for corank--1 germs (that is, $\rk(d \Phi_0)=1$) with $T(\Phi)=0$ in algebraic way. We present a brief summary of this case.

 Those germs have a normal form $\Phi(s, t) = ( s, t^2, t d(s,t) ) $, where $d(s,t) = h(s, t^2) $ for  some germ $h$, such that $d(s,t)$  is not divisible by $t$. In this case the equation of the image is $g(x,y,z)= yh^2(x,y) - z^2 =0 $. See \cite{Mond0,nunodoodle} for details.

 For these germs, the equation of $D$ is $d$, the involution is $\iota(s,t)=(s,-t)$, and 
 \[\Sigma=\{(x,y, z) \in \C^3 \ | \ z=h(x,y)=0 \} \] 
 
 The transverse section $H$ can be chosen as $H(x,y,z)=z$, cf. Preliminaries (Subsection~\ref{ss:boundmilnor}). With this choice the $H$-vertical indices $\mathfrak{v}_j$ are 0. This fact follows from the product decomposition
\csere{\begin{equation}\label{eq:splittingspec}
    g(x,y,z)=(\sqrt{y} \cdot h(x,y) -z) \cdot (\sqrt{y} \cdot h(x,y) +z)
\end{equation}}
 around $\Sigma \setminus \{0\}$, for which the $\beta_{1,2}(\tau)$ coefficients in equation~\eqref{eq:beta} are constant functions.

 The pull-back of $H$ is $(H \circ \Phi)(s,t)=td(s,t)$, hence $D_{\sharp}=\{t=0\}$, and
 \[ \lambda_i=-\sum_{k \neq i} D_i \cdot D_k -D_i \cdot \{t=0 \},\]
 hence the vertical indices are
$\mathfrak{vi}_j=\lambda_i+\lambda_{\sigma(i)}$ for untwisted components, and $\mathfrak{vi}_j=\lambda_i$ for twisted components. Their sum is
\[\sum_{j \in J} \mathfrak{vi}_j=\sum_{i=1}^l \lambda_i=-\sum_{k \neq i} D_i \cdot D_k- D \cdot \{t=0\}.\]
The last term is
\[D \cdot \{t=0\}=\dim \frac{\mathcal{O}_{(\C^2,0)}}{(t, d(s,t))}=C(\Phi),\]
cf. Preliminaries (Subsection~\ref{ss:invfindet}), since in this case the ideal $(t,d(s,t)) \subset \mathcal{O}_{(\C^2,0)}$ is equal to the ramification ideal, the ideal generated by the $2 \times 2$ minors of the Jacobian $d \Phi$ of $\Phi$. Altogether we obtain
\[\sum_{j \in J} \mathfrak{vi}_j=-\sum_{k \neq i} D_i \cdot D_k- C(\Phi),\]
what we wanted to prove.

\subsection{Examples}\label{ss:examp}

The boundary of the Milnor fibre is presented for several examples  in \cite{NP2}. Unfortunately the vertical indices are not provided there, since the article focuses on the plumbing graph $\partial F$, which requires a modified version of $\mathfrak{vi}_j$. Here we provide the vertical indices and verify that equation~\eqref{eq:osszeg} holds. \csere{The signature of the Milnor fiber for these examples is computed in \cite{GCP}, and this computation also uses the vertical indices.}

Except for the last one, all examples are taken from
D. Mond's list  of simple germs  \cite[Table 1]{Mond2}, and among them only the family $H_k$ has nonzero triple point number $T$, hence the others are examples for the simplest case described in Subsection~\ref{ss:cor1T0}. The last example is a corank-2 germ from \cite{marar}.

\begin{ex}[The family $S_{k-1}$ ($k \geq 2$)]\label{ex:S}

\[  \Phi(s, t) = ( s, t^2, t^3 + s^k t ) \mbox{, and }
 g(x, y, z) = y(y+x^k)^2 - z^2  ,\]
 \[d(s,t)=t^2+s^k,\]
 \[C(\Phi)=k, \ T(\Phi)=0.\]
 
 \emph{If $k=2n+1$ is odd,} then $d$ is irreducible. Hence there is only one twisted component with $\mathfrak{vi}=-k$, which is equal to $-C(\Phi)$.

 \emph{If $k=2n$ is even,} then 
 \[d(s,t)=(s^n-it)(s^n+it),\]
and  the factors define the two preimages of an untwisted component, since $\Sigma$ has only one component, because $h(x,y)=y+x^k$ is irreducible. $\lambda_1=\lambda_2=-n-n=-k$, hence $\mathfrak{vi}=-2k$.
Equation~\eqref{eq:osszeg} holds as
\[-2k=\mathfrak{vi}=-(D_1 \cdot D_2+D_2 \cdot D_1)-C(\Phi)+3T(\Phi)=-(n+n)-k+0.\]

\csere{Although this example fits into the special case discussed in Subsection~\ref{ss:cor1T0}, with the computation of the vertical indices is simplified and avoids the Taylor series computations in equation~\eqref{eq:beta}, for the reader's convenience we illustrate this latter method on this example. }

\csere{
\begin{equation}
    \Sigma=\Phi(D)=\{(x,y,z) \in \C^3 \ | \ z=0 \text{ and } y+x^k=0 \},
\end{equation}
it has only one component. Its parametrization is given by
\begin{equation}
    p(\tau)=(\tau, -\tau^k, 0)
\end{equation}
}

\csere{The factorization $g=g_1 \cdot g_2$ described in equation~\eqref{eq:splittingspec} determines the factorization $T_2(g)(p(\tau))=T_1(g_1)(p(\tau)) \cdot T_1(g_2)(p(\tau))$ of the quadratic Taylor polynomial of $g$ centered at $0 \neq p(\tau) \in \Sigma$ with variables $x'=x-\tau$, $y'=y+\tau^k$ and $z'=z$. Namely, we have
\begin{equation}
    T_2 (g)(p(\tau))(x',y',z')=(ik\tau^{\frac{3k-2}{2}}x'+ik\tau^{\frac{k}{2}}-z')(ik\tau^{\frac{3k-2}{2}}x'+ik\tau^{\frac{k}{2}}+z')
\end{equation}
}

\csere{We compute the vertical indices with two different choices of the transverse section $H$.}

\csere{First we choose $H_1(x,y,z)=z$, as Subsection~\ref{ss:cor1T0} suggests. Then 
\begin{equation}
    T_1(H_1)(p(\tau))=-\frac{1}{2} T_1(g_1)(p(\tau))+\frac{1}{2} T_1(g_2)(p(\tau)),
\end{equation}
hence, $\beta_1(\tau)=-\frac{1}{2}$ and $\beta_2(\tau)=\frac{1}{2}$ are constant functions, implying that $\mathfrak{v}=\text{ord}(\beta_1 \beta_2)=0$ when calculated using $H_1$.
}

\csere{The second choice is $H_2(x,y,z)=y+x^k$. For this we have
\begin{equation}
    T_1(H_2)(p (\tau))=-i\tau^{-\frac{k}{2}} T_1(g_1)(p(\tau))-i\tau^{-\frac{k}{2}} T_1(g_1)(p(\tau)). 
\end{equation}
Although $\beta_1(\tau)=\beta_2(\tau)=-i\tau^{-\frac{k}{2}}$ are not globally defined meromorphic functions if $k$ is odd, their product $\beta_1(\tau)\beta_2(\tau)=-\tau^{-k}$ is always globally defined of order $\mathfrak{v}=-k$ when the calculation uses $H_2$.
}

\csere{We see that the vertical index $\mathfrak{v}$ with respect to $H$ really does depend on the choice of $H$. But $\mathfrak{vi}$ does not, as we can verify in this example. Let us calculate the terms $\lambda$ with the two choices $H_1$ and $H_2$! }

\csere{The pull-back of $H_1$ is
\begin{equation}
H(\Phi(s,t))=t(s^k+t^2)=td(s,t) 
\end{equation}
 in accordance with Subsection~\ref{ss:cor1T0}. The extra curve is $D_{\sharp}=\{t=0\}$. This gives the $\lambda$ terms and the vertical indices we computed above, namely, if $k=2n+1$ is odd, then
\begin{equation}
\mathfrak{vi}=\lambda+\mathfrak{v}=D \cdot D_{\sharp}+0=-k, 
\end{equation}
and if $k=2n$ is even, then
\begin{equation}
\lambda_1=-D_1 \cdot D_2-D_1 \cdot D_{\sharp}=-k \text{, similarly } \lambda_2=-k \text{, and } \mathfrak{vi}=\lambda_1+\lambda_2+\mathfrak{v}=-2k.
\end{equation}
}

\csere{
The pull-back of $H_2$ is
\begin{equation}
H(\Phi(s,t))=s^k+t^2=d(s,t), 
\end{equation}
hence the extra curve $D_{\sharp}$ is empty with the choice $H_2$. The $\lambda$ terms and the vertical index in case are the following:
if $k=2n+1$ is odd, then $\lambda=0$, and
\begin{equation}
\mathfrak{vi}=\lambda+\mathfrak{v}=0-k=-k, 
\end{equation}
and if $k=2n$ is even, then
\begin{equation}
\lambda_1=\lambda_2=-D_1 \cdot D_2=-n \text{, and } \mathfrak{vi}=\lambda_1+\lambda_2+\mathfrak{v}=-n-n-k=-2k.
\end{equation}
We see that the vertical index $\mathfrak{vi}$ is the same with the choices $H_1$ and $H_2$.
}

\csere{Note that for the germs discussed in Subsection~\ref{ss:cor1T0}, two special choices $H_1(x,y,z)=z$ and $H_2(x,y,z)=h(x,y)$ are always possible. Their pull-back germs are $H_1(\Phi(s,t))=td(s,t)$, implying that $D_{\sharp}=\{t=0\}$ and $\mathfrak{v}_j=0$ with this choice, and $H_2(\Phi(s,t))=d(s,t)$, hence $D_{\sharp}$ is empty with this choice.}

\end{ex}

\begin{ex}[The family $B_{k} $ ($k\geq 1$)]
\[ \Phi(s, t) = ( s, t^2, s^2t + t^{2k+1} ) \mbox{, and }
 g(x,y,z)= y(x^2+y^k)^2 - z^2  ,\]
 \[d(s,t)=s^2+t^{2k}=(s+it^k)(s-it^k),\]
 \[C(\Phi)=2, \ T(\Phi)=0.\]
 
 Although $D$ always has two components, the number of the components of $\Sigma$ depends on the parity of $k$.

 \emph{If $k=2n+1$ is odd,} then $\Sigma$ has one component, since $h(x,y)=x^2+y^k$ is irreducible. This component is untwisted with preimages $D_1$ and $D_2$. 
 Then $\lambda_1=\lambda_2=-k-1$ and $\mathfrak{vi}=-2k-2$. 
 Equation~\eqref{eq:osszeg} holds as
\[-2k-2=\mathfrak{vi}=-(D_1 \cdot D_2+D_2 \cdot D_1)-C(\Phi)+3T(\Phi)=-(k+k)-2+0.\]

\emph{If $k=2n$ is even,} then $h(x,y)=(x-iy^n)(x+iy^n)$, hence $\Sigma$ has two components, $\Sigma_1$ is twisted with preimage $D_1$ and $\Sigma_2$ is twisted with preimage $D_2$. Then 
\[\lambda_1=\mathfrak{vi}_1=\lambda_2 =\mathfrak{vi}_2=-k-1,\]
and we check that
\[ \mathfrak{vi}_1+\mathfrak{vi}_2=-(D_1 \cdot D_2+D_2 \cdot D_1)-C(\Phi)+3T(\Phi)=-2k-2.\]

\end{ex}

\begin{ex}[The family $C_k $ ($k\geq 1$)]
\[\Phi(s, t) = ( s, t^2, st^3 + s^k t ) \mbox{, and }
 g(x,y,z) = y(xy+x^k)^2 - z^2 ,\]
 \[d(s,t)=st^2+s^k,\]
 \[C(\Phi)=k , \ T(\Phi)=0.\]

 Since $h(x,y)=x(y+x^{k-1})$, $\Sigma$ has two components, \[\Sigma_A=\{x=z=0\} \mbox{, and } \Sigma_B=\{y+x^{k-1}=z=0\} .\]

\emph{If $k=2n+1$ is odd,} then 
\[d(s,t)=s(t+is^n)(t-is^n)\]
has three components. $D_1 \to \Sigma_A$ is a twisted component, $D_2 \cup D_3 \to \Sigma_B$ is an untwisted component. 
\[\lambda_1=-1-1-1=-3 \mbox{, and } \lambda_2=\lambda_3=-n-n-1=-k, \]
hence $\mathfrak{vi}_A=-3$ and $\mathfrak{vi}_B=-2k$. The left hand side of equation~\eqref{eq:osszeg} is
$ \mathfrak{vi}_A + \mathfrak{vi}_B=-2k-3$, while the right hand side is
\[-(2D_1 \cdot D_2+2D_1 \cdot D_3 +2D_2 \cdot D_3)-C(\Phi)+3 T(\Phi)=-(2+2+2n)-k=-2k-3. \]
 
\emph{If $k=2n$ is even,} then $D$ has two components, $D_1 \to \Sigma_A$ and $D_2 \to \Sigma_B$ are twisted components. $\lambda_1=\mathfrak{vi}_A=-3$ and $\lambda_2=\mathfrak{vi}_B=-(k-1)-2=-k-1$. The sum is 
\[\mathfrak{vi}_A + \mathfrak{vi}_B=-2D_1 \cdot D_2-C(\Phi)+3 T(\Phi)=-k-4.\]

\end{ex}

\begin{ex}[The family $H_{k} $ ($k\geq 1$)]\label{ex:H}
\[ \Phi(s, t) = ( s, st + t^{3k-1}, t^3 ) \mbox{, and }  g(x, y, z) = z^{3k-1} - y^3 + x^3 z + 3 xyz^k \]
\[d(s,t)=(s-\rho^2 t^{3k-2})(s-\rho t^{3k-2}) \mbox{, where } \rho=-\frac{1}{2}+\frac{\sqrt{3}}{2}i.\]
\[C(\Phi)= 2, \ T(\Phi)=k-1\]

The equation $g$ of the image of $\Phi$ and the triple point number $T(\Phi)$ can be computed by using fitting ideals, see e.g. Example 1.3.4. in \cite{gtezis}. $\Sigma$ has one component, hence it is \csere{untwisted}. This follows from the fact that the involution
\[\iota(s,t)=
\left\{ \begin{array}{ccc}
(s, \rho t) & \mbox{if} & (s,t) \in D_1 \\
(s,\rho^2 t) & \mbox{if} & (s,t) \in D_2 \\
\end{array} \right.
\]
maps $D_1$ to $D_2$ and $D_2$ to $D_1$. The computation of the vertical index is much more complicated than in the previous examples, it was done in \cite{NP2, gtezis}, and the result is $\mathfrak{vi}=-3k-1$. Equation~\eqref{eq:osszeg} reads as
\[-3k-1=\mathfrak{vi}=-2D_1 \cdot D_2-C(\Phi) +3 T(\Phi)=-2(3k-2)-2+3(k-1).\]

\end{ex}

\begin{ex}[A corank--2 germ from \cite{marar}] 
\[ \Phi (s, t) = (s^2, t^2, s^3 + t^3+ st) ,\]
\[g(x, y, z) = x^6-2x^4y + x^2y^2 - 2x^3y^3-2xy^4 + y^6 - 8x^2y^2z - 2x^3z^2-2xyz^2-2y^3z^2+z^4  ,\]
\[d(s,t)=\prod_{i=1}^5d_i(s,t)=(s+t^2)(s^2+t)(s+t)(s+\rho t)(s+\rho^{-1}t) \mbox{, where } \rho=-\frac{1}{2}+\frac{\sqrt{3}}{2}i.\]
\[C(\Phi)=3, \ T(\Phi)=1.\]

All the five components are twisted, since each $D_i$ is invariant under the involution, which is $\iota(s,t)=(s,-t)$ on $D_1$, $\iota(s,t)=(-s,t)$ on $D_2$, and $\iota(s,t)=(-s,-t)$ on $D_3$, $D_4$ and $D_5$. Each vertical index is $
\mathfrak{vi}_j=-4$. Equation~\eqref{eq:osszeg} holds with
\[ -20=\sum_{j=1}^5 \mathfrak{vi}_j=-\sum_{i \neq k} D_i \cdot D_k -C(\Phi)+3T(\Phi)=-5 \cdot 4-3+3.\]

    
\end{ex}

\section{Final remark and open questions}\label{s:que}

By Corollary~\ref{co:vitopinv}, the topological reformulation of the vertical indices implies their topological invariance. Assuming the results of a work in progress where the signature of the Milnor fibre is computed in terms of $C(\Phi)$, $T(\Phi)$ and $\mathfrak{vi}_j$, we would get the topological invariance of the signature $\sigma(F)$ as well.

\begin{question}
How to give a purely algebraic description of the vertical indices $\mathfrak{vi}_j$~? 
\end{question}
\noindent Although the original definition is computable, it is not exactly algebraic as it uses an aid-germ $H$ and a splitting of the second order term of the equation $g$ along $\Sigma \setminus \{0 \}$, which is complicated in general.

\begin{question}
What is the vertical monodromy of the cylinder fibre of $\partial F$ around the components $\Upsilon_j$? The vertical mondoromy is an authomorphism of the relative homology $H_1(S^1 \times I, \partial(S^1 \times I);\Z)$, it is an iterated Dehn twist. Is the number of Dehn twists related to $\mathfrak{vi}_j$?
\end{question}

\begin{question}
Which complex plane curves $(D, 0) \subset (\C^2, 0)$ can appear as the double point set of a finitely determined germ $\Phi: (\C^2, 0) \to (\C^3, 0)$? What are the possible involutions $\iota$ on $D$? The examples suggest that the pairs $D_i$ and $D_{\sigma(i)}$ are topologically equivalent, moreover they are equivalent relative to the other components as well. Does this hold in general?

The same can be asked about stable immersions $f:S^3 \looparrowright \R^5$ as well. Which links $\gamma \subset S^3$ can appear as the double point set of $f$?
\end{question}

\renewcommand{\refname}{References}

\end{document}